\documentclass[11pt]{article}

\usepackage{amsmath,amsthm,amssymb}
\usepackage[T1]{fontenc}
\usepackage[utf8]{inputenc}
\usepackage[dvipdf]{graphicx}
\usepackage{color}
\usepackage{epstopdf}
\usepackage{dsfont}
\usepackage{enumerate}
\usepackage{enumitem}

\definecolor{db}{RGB}{0, 0, 130}
\usepackage[colorlinks=true,citecolor=red,linkcolor=db,urlcolor=blue,pdfstartview=FitH]{hyperref}

\usepackage[normalem]{ulem}
\usepackage[numbers]{natbib}

\definecolor{rp}{rgb}{0.25, 0, 0.75}
\definecolor{dg}{rgb}{0, 0.6, 0}

\newtheorem{theorem}{Theorem}[section]

\newtheorem{definition}{Definition}[section]

\newtheorem{example}[definition]{Example}
\newtheorem{assumption}[theorem]{Assumption}
\newtheorem{lemma}[definition]{Lemma}

\newtheorem{remark}[definition]{Remark}
\def\1{\mathbf{1}}
\def\R{\mathbb{R}}

\def\E{\mathbb{E}}

\def\L{\mathbb{L}}
\def\F{\mathbb{F}}
\def\P{\mathbb{P}}

\def\G{\mathbb{G}}

\def\Bc{\mathcal{B}}
\def\Cc{\mathcal{C}}
\def\Fc{\mathcal{F}}

\def\Pb{\widehat{\P}}

\def\Gc{\mathcal{G}}
\def\Lc{\mathcal{L}}
\def\Pc{\mathcal{P}}
\def\Uc{\mathcal{U}}

\def\Kc{\mathcal{K}}

\def\xb{\mathbf{x}}

\def\x{\times}
\def\Om{\Omega}
\def\om{\omega}
\def\eps{\varepsilon}

\def\Omb{\overline \Om}
\def\Fcb{\overline \Fc}
\def\Pcb{\overline \Pc}
\def\Pb{\overline \P}
\def\Fb{\overline \F}
\def\Ph{\widehat{\P}}
\def\Wc{\mathcal{W}}
\def\Mb{\overline M}
\def\Jb{\overline J}

\def\Ur{{\rm U}}
\def\Uc{{\cal U}}
\def\ur{{\rm u}}

\title{A $C^1$-It\^o's formula for flows of semimartingale distributions
	\thanks{We are grateful to Pierre Cardaliaguet and Chenchen Mou for helpful discussions.}}

\author{
Bruno Bouchard
\footnote{CEREMADE, Universit\'e Paris Dauphine - PSL, CNRS.  bouchard@ceremade.dauphine.fr. }
\and
Xiaolu Tan
\footnote{Department of Mathematics, The Chinese University of Hong Kong. xiaolu.tan@cuhk.edu.hk. Research supported by CUHK startup grant, and by Hong Kong RGC General Research Fund (Projects 14302921 and 14302622).}
\and
Jixin Wang
\footnote{Department of Mathematics, The Chinese University of Hong Kong. jxwang@math.cuhk.edu.hk.}
}

\date{\today}

\begin{document}

\maketitle

\begin{abstract}
	We provide an It\^o's formula for $C^1$-functionals of flows of conditional marginal distributions of  continuous semimartingales.
	This is based on the notion of  weak Dirichlet process, and extends the $C^1$-It\^o's formula in Gozzi and Russo (2006) to this context.
	As the first application, we study a class of McKean-Vlasov optimal control problems,
	and establish a verification theorem which only requires $C^1$-regularity of its value function, which is equivalently the (viscosity) solution of the associated HJB master equation.
	It goes together with a novel duality result.
\end{abstract}

\section{Introduction}

	In its classical formulation, It\^o's formula  provides a canonical decomposition for $C^2$-transformations of  semimartingales.
	Since it was introduced, various variations have been proposed.
	In particular, the $C^1$-It\^o's formula, that was developed in the series of works \cite{Russo and Vallois1, Russo and Vallois2, Russo and Vallois3, Gozzi and Russo} using the notion of weak Dirichlet process and the stochastic calculus via regularization approach, only requires the transformation to be $C^1$.
	In this theory, a $C^1$-functional of a weak Dirichlet process is again a weak Dirichlet process, which can be (uniquely) decomposed as the sum of a martingale  and an orthogonal process.
	%In many cases, see e.g. \cite{Bouchard Tan} and \cite{Gozzi and Russo}, the $C^1$-It\^o's formula is enough to solve   {concrete problems}, as it allows one to identify the martingale part of  {the} considered process.

	\vspace{0.5em}
	
	Recently, motivated by the study of mean-field problems, involving McKean-Vlasov processes, or Mean-Field Games (MFG),
	an It\^o's formula for flows of semimartingale distributions has been introduced, see e.g. \cite{BuckdahnLiPeng,Lions, Carmona and Delarue book} or the recent paper \cite{Fadel Touzi} and the references therein.
	It applies to transformations of measure-valued processes, obtained as the flows of (conditional) marginal distributions of semimartingales,
	and provides a decomposition of $C^2$-functionals of such measure-valued processes.
	In particular, it can be used to deduce the master equation for  MFGs, or the Hamilton-Jacobi-Bellman (HJB) equation of  McKean-Vlasov control problems.
	However, in practical  situations of application, it is usually not  easy to check the required $C^2$-differentiability of the value function defined on the space of probability measures.

	\vspace{0.5em}

	In this paper, we provide a $C^1$-It\^o's formula for flows of semimartingale distributions,
	by using the notion of   weak Dirichlet process as in Gozzi and Russo \cite{Gozzi and Russo}.
	This requires less regularity   on the value function and turns out to be   enough in many  applications.
	More precisely, on a filtered probability space $(\Om, \Fc, \F =(\Fc_t)_{t \ge 0}, \P)$, let $X$ be a $\R^d$-valued continuous semimartingale with decomposition
	$$
		X_t = X_0 + A_t + M_t + \int_0^t \sigma^{\circ}_s d M^{\circ}_s,
	$$
	where $A$ is a continuous finite variation process, $\sigma^{\circ}$ is progressively measurable, and both $M$ and $M^{\circ}$ are continuous martingales.
	Let $\Gc := \sigma(M^{\circ}_t, t \ge 0)$ denote the sub-$\sigma$-field generated by $M^{\circ}$, which is usually referred to as the common noise $\sigma$-field in the mean-field literature.
   Then, one defines a process $m = (m_t)_{t \ge 0}$, taking values in the space $\Pc(\R^d)$ of probability measures on $\R^d$, by
	$$
		m_t := \Lc(X_t \big| \Gc),
		~~t \ge 0.
	$$
	Besides, we consider a continuous weak Dirichlet process $Y$ with (unique) decomposition
	\begin{equation*}
		Y_t=Y_0+ A^Y_t + M^Y_t,
	\end{equation*}
	where $M^Y$ is the martingale part  {of $Y$}, and $A^Y$ is its orthogonal part (see Section \ref{sec:Ito} for a precise definition).
	Under some (essentially related to integrability) technical conditions, we prove that,
	for a continuous function $F: (t,y,m) \in \R_+ \x \R^d \x \Pc (\R^d) \longrightarrow F(t,y,m) \in \R$
	with continuous partial derivative $D_y F$  {in} $y$, and continuous intrinsic derivative $D_m F$ (see Section \ref{sec:Ito} for a precise definition),
	one has
		\begin{align*}
			F(t, Y_t, m_t)
			=
			F(0,Y_0,m_0)
			& +
			\int_{0}^{t} D_y F(s,Y_s,m_s)  ~d M^Y_s  \nonumber \\
			&+
			\int_{0}^{t} \E \left[D_m F(s,y,m_s,X_s) \sigma_s^{\circ} \Big| M^{\circ}\right]_{y=Y_s}
		dM_s^{\circ}
			+
			\Gamma_t,
		\end{align*}
	where $\Gamma$ is an orthogonal process.
	When $F$ is a $C^2$-functional, $\Gamma$ can be explicitly expressed  {in terms of the} first order time derivative $\partial_t F$, together with the second order derivatives $D^2_y F$, $D_x D_m F$ and $D^2_m F$ (see e.g. \cite[Section 6]{Carmona and Delarue}).
	In particular, we extend the It\^o's formula for $C^1$-transformation of weak Dirichlet processes of \cite{Gozzi and Russo} to our context.
\\	

	Importantly, this formula allows   one to identify the martingale part of the process $F(\cdot,Y, m)$,
	which is   enough in many practical situations of   application.
	Let us for instance refer to \cite{Gozzi and Russo} for an application {in optimal control},
	and to \cite{Bouchard Tan}  for some applications in mathematical finance.
\vspace{2mm}

The second contribution of this paper is to provide a new type of application in the form of a verification argument.
Namely, we consider a McKean-Vlasov optimal control problem of the form:
	$$
		\sup_{\nu \in \Uc}
		~
		\E \Big[ \int_0^T L(t, \rho^{\nu}_t, \nu_t) dt + g(\rho^{\nu}_T) \Big],
	$$
	where, given two independent Brownian motions $W$ and $W^{\circ}$,
	$X^{\nu}$ is defined by the controlled McKean-Vlasov SDE:
	$$
		    X_t^{\nu}
		    =
		    X_0
		    +
		    \int_0^t  \sigma(s, X^{\nu}_s, \rho^{\nu}_s) dW_s
		    +
		    \int_0^t  \sigma_0(s, X^{\nu}_s, \rho^{\nu}_s) \big( dW_s^{\circ} + \nu_s ds \big),
		    ~~\mbox{with}~~
		    \rho^{\nu}_t := \Lc(X^{\nu}_t | W^{\circ}),
	$$
	and where an admissible control process $\nu \in \Uc$ is a $\F^{W^{\circ}}$-progressively measurable process taking value in a compact set ${\Ur} \subset \R^d$.
	In  {the} above, $W^{\circ}$ is the so-called common noise, and $\F^{W^{\circ}}$ denotes the filtration generated by $W^{\circ}$.
	Notice that one only controls the drift process, and the control depends only on the common noise $W^{\circ}$.
	For this McKean-Vlasov control problem, the value function can be written as $V(t, m)$,
	where $V$ is the unique (viscosity) solution of the master HJB equation (see e.g. Pham and Wei \cite{Pham and Wei}):
	$$
		\partial_t V(t, m) + \L[V]  + H \big( t,m, D_m V(t,m,\cdot) \big) = 0,
	$$
	where $\L[V]$ is a linear operator  {involving   $D_x D_m V$} and $D^2_m V$ (see Section \ref{sec:Verification} for an explicit expression), and $H$ is the Hamiltonian given by
	$$
		H(t, m, D_m V(t, m, \cdot)) := \sup_{{u \in \Ur}} \Big( L(t, m, {u}) + {{u}  m} \big( \sigma_0(t, \cdot, m) D_mV(t, m, \cdot) \big) \Big).
	$$
	The classical verification theorem states that, given a smooth solution to the HJB equation (or equivalently the value function),
	the optimizer in the definition of $H$ provides a feedback optimal control.
	It relies on   It\^o's formula,   assuming that $V \in C^{1,2}$ in the sense that $V$, $\partial_t V$, $D_m V$, $D_x D_m V$ and $D^2_m V$ are all  well-defined and continuous (see e.g. \cite{Guo Pham Wei} and \cite{Talbi Touzi Zhang} for some closely related problems).
	At the same time, for the above class of optimal control problems,
	the definition of the Hamiltonian $H$ as well as the associated optimizer only involve  the first order derivative $D_m V$.
	It is then natural to ask whether it is enough to only require $C^{1}$-regularity on $V$  {(in the sense that $V$ and $D_m V$ are both continuous)}.
\\
By using our $C^1$-It\^o's formula, we actually  establish a verification theorem which only assumes that $V$ is ${C^{1}}$. 	To the best of our knowledge,  {this approach is new even for} classical optimal control problems. The proof goes together with the proof of a dual formulation which is of own interest.

	\vspace{0.5em}

	The rest of this paper is organized as follows. The $C^1$-It\^o's formula for flows of semimartingale distributions is proved in Section \ref{sec:Ito}.
	 Section \ref{sec:Verification} is dedicated to our verification and duality arguments for a class of McKean-Vlasov optimal control problems.
%	In the Appendix, we illustrate the main idea of our approach with $C^1$-It\^o's formula, by considering a class of classical optimal control problems.

\section{A \texorpdfstring{$C^1$}{}-It\^o's formula for flows of semimartingale distributions}
\label{sec:Ito}

	Throughout the section, we fix a complete probability space $(\Om, \Fc, \P)$, equipped with a filtration $\F = (\Fc_t)_{t \ge 0}$ satisfying the usual conditions.
	We assume that $\Fc$ is countably determined.

\subsection{Preliminaries}

	Let $\Pc(\R^d)$ denote the space of all (Borel) probability measures on $\R^d$,
	and $\Pc_2(\R^d)$ denote the set of all probability measures on $\R^d$ with finite second moment,
	i.e.
	$$
		\Pc_2(\R^d) = \Big\{\mu \in \Pc(\R^d) : \int_{\R^d}|x|^2 \mu(d x) <\infty \Big\}.
	$$
	The space $\Pc_2(\R^d)$ is equipped with the Wasserstein distance
	$$
		\Wc_2(\mu_1, \mu_2) := \left(\inf_{\gamma \in \Gamma(\mu_1,\mu_2)} \int_{\R^d \x \R^d} |x-y|^2 d\gamma(x,y)\right)^{\frac{1}{2}},
	$$
	where $\Gamma(\mu_1,\mu_2)$ is the set of all couplings of $\mu_1$ and $\mu_2$, i.e.~joint probability measures on $\R^d \x \R^d$ whose marginals are $\mu_1$ and $\mu_2$, respectively.

	\begin{definition}\cite[Definition 5.43]{Carmona and Delarue book}\label{functional linear derivative}
		A function $F: \Pc_2(\R^d) \longrightarrow \R$ is said to have a linear functional derivative if there exists a function
		$$
			\frac{\delta F}{\delta m}: \Pc_2(\R^d) \times \R^d \ni(m, x) \longmapsto \frac{\delta F}{\delta m}(m)(x) \in \R,
		$$
		 {that is} continuous for the product topology,    {such that,} for any bounded subset $\Bc \subset \Pc_2(\R^d)$, the function $\R^d \ni x \longmapsto$ $[\delta F / \delta m](m)(x)$ is at most of quadratic  {growth  uniformly} in  {$m \in \Bc$}, and such that, for all $m$ and $m^{\prime}$ in $\Pc_2(\R^d)$:
		\begin{equation} \label{eq:def_dmF}
			F(m^{\prime})-F(m)=\int_0^1 \int_{\R^d} \frac{\delta F}{\delta m}\left(t m^{\prime}+(1-t) m\right)(x) \left[m^{\prime}-m\right](d x) d t .
		\end{equation}
		Assume further that, for any $m \in \Pc_2(\R^d)$, the function $\R^d \ni x \longmapsto {\frac{\delta F}{ \delta m}}(m)(x)$ is differentiable. Then, one defines the intrinsic derivative
		\begin{equation} \label{eq:def_DmF}
			D_m F(m, x):= {D_x} \delta_m F (m, x),
			~~\mbox{with}~~
			\delta_m F(m, x) := \frac{\delta F}{\delta m}(m)(x),
		\end{equation}
	in which $D_x$ is the gradient operator with respect to the $x$-variable.
	\end{definition}

	In our paper, we will stay in the setting where $D_m F(m,x)$ is jointly continuous in $(m,x)$ and is at most of linear growth in $x$, uniformly in $m \in \Bc$, for any bounded subset $\Bc \subset \Pc_2(\R^d)$.
	Then, for any $m \in \Pc_2(\R^d)$, the function $\R^d \ni x \longmapsto D_m F(m,x)$ is uniquely defined $m$-almost everywhere on $\R^d$.

	\vspace{0.5em}

	We shall make use of the notion of weak Dirichlet process and stochastic calculus by regularization,
	for which we now define the  notions of quadratic variation and of orthogonal process (see e.g. \cite[Definition 3.4]{Gozzi and Russo}). Recall that a sequence of stochastic processes
	$\big\{ (X_t^n)_{t\ge 0}, n \ge 1\big\}$   is said to converge to the process $(X_t)_{t \ge 0}$ in the u.c.p topology (uniform convergence on compacts in probability) if,
	for all $\eps > 0$ and $t \ge 0$,
	$$
		\lim_{n \to \infty} \P \Big[ \sup_{s \le t} |X_s^n-X_s|> \varepsilon \Big] = 0.
	$$

	\begin{definition} \label{def:weakDiri}
		$\mathrm{(i)}$ Given two càdlàg processes $X$ and $Y$, the quadratic covariation $[X,Y]$ is defined by
		$$
			[X,Y]_s = \lim_{\varepsilon \to 0}
			\frac{1}{\varepsilon}
			\int_0^s (X_{r+\varepsilon} - X_r)(Y_{r+\varepsilon} - Y_r)dr, ~s \ge 0,
		$$
		if the limit exists in the sense of u.c.p.

		\vspace{0.5em}
		
		\noindent $\mathrm{(ii)}$ Let $A$ be a $\F$-adapted c\`adl\`ag process, we say  {that} $A$ is an orthogonal process if $\left[A, N\right] = 0$ for every continuous $\F$-local martingale $N$.
		
		\vspace{0.5em}
		
		\noindent $\mathrm{(iii)}$ Let $X$ be a $\F$-adapted c\`adl\`ag process, it is called a weak Dirichlet process if it has the decomposition
		$$
			X_t = X_0 + A_t + M_t, ~~t \ge 0,
		$$
		where $M$ is a local martingale, and $A$ is an orthogonal process w.r.t. the filtration $\F$.
	\end{definition}

	\begin{remark}
		$\mathrm{(i)}$
		When $X$ and $Y$ are c\`adl\`ag semimartingales, $[X,Y]$ coincides with the usual bracket (see e.g. \cite[Proposition 1.1]{Russo and Vallois2}).
		
		\vspace{0.5em}
		
		\noindent $\mathrm{(ii)}$ In the definition of the orthogonal process, it is equivalent to consider all bounded continuous martingales $N$ in place of all continuous local martingales.
		
		\vspace{0.5em}
		
		\noindent $\mathrm{(iii)}$
		For a continuous weak Dirichlet process, its decomposition as the sum of an orthogonal process and a local martingale is unique.
		
		\vspace{0.5em}
		
		\noindent $\mathrm{(iv)}$ We will consider later a sub-filtration $\G$ of $\F$. Nevertheless, throughout the paper, the notion of   orthogonal   and   weak Dirichlet process are all w.r.t.~the filtration $\F$.
	\end{remark}

\subsection{Main Result}

	 From now on, we fix a continuous semimartingale $\left(X_t\right)_{t \ge 0}$ on the filtered space $(\Om, \Fc, \F, \P)$ with decomposition
	\begin{equation} \label{eq:structure_X}
		X_t=X_0+ A_t + M^X_t,
		~~\mbox{with}~~
		M^X_t = M_t +  \int_0^t \sigma^{\circ}_s d M^{\circ}_s, ~t \ge 0,
	\end{equation}
	where $(A_t)_{t \ge 0}$ is   continuous  {with} finite variation,
	$(M_t)_{t \ge 0}$ and $(M_t^{\circ})_{t \ge 0}$ are continuous martingales,
	with $A_0 = M_0 = M_0^{\circ} =0$.
	Let us define $\G = (\Gc_t)_{t \ge 0}$ as the (raw) filtration generated by $M^{\circ}$, i.e.
	$$
		\Gc_t := \sigma(M^{\circ}_s, ~ 0 \le s \le t), ~~t \ge 0,
	$$
	and define
	$$
		\Gc := \sigma(M^{\circ}_s, ~s \ge 0).
	$$
	\begin{assumption}\label{assum:main}
		$\mathrm{(i)}$ The process $\sigma^{\circ}$ is $\F$-progressively measurable,
		and there exists an increasing sequence of stopping times $(\tau_n)_{n \ge 1}$ w.r.t. $\G$ such that
		$\tau_n \longrightarrow \infty$, a.s. as $n \longrightarrow \infty$, and
		\begin{equation} \label{eq:SquareIntX}
			\E \Big[ [ M ]_{\tau_n \wedge t} + \big|A \big|_{\tau_n \wedge t}^2  +
				\int_0^{\tau_n \wedge t}  \big|\sigma^{\circ}_{s} \big|^2 d [M^{\circ} ]_s
			\Big] <+\infty, ~\mbox{for all}~t \ge 0~\mbox{and}~n \ge 1,
		\end{equation}
		where $\left(\left|A\right|_t\right)_{t \ge 0}$ denotes the total variation of $A$.
		
		\vspace{0.5em}

		\noindent $\mathrm{(ii)}$ The martingale $M$ is orthogonal to $N$ (i.e. $[M, N] = 0$), for all (c\`adl\`ag) $\G$-martingale $N$.

		\vspace{0.5em}
		
		\noindent $\mathrm{(iii)}$ $(H)$-hypothesis condition:
		$$
			\E \big[\mathbf{1}_D \big| \Gc_t \big] = \E \big[\mathbf{1}_D \big| \Gc \big], ~\mbox{a.s., for all} ~ D \in \Fc_t, ~t \ge 0.
		$$
	\end{assumption}
	
	 {We next introduce the}  $\Pc(\R^d)$-valued process $m = (m_t)_{t \ge 0}$  {associated to the $\Gc$-conditional law of $X$:}
	$$
		m_t := \Lc \left(X_t ~|\Gc_t \right) = \Lc( X_t ~| \Gc),
		~~t \ge 0.
	$$
	Assumption \ref{assum:main} ensures that $m = (m_t)_{t \ge 0}$ is continuous under $\Wc_2$, as shown in the following lemma.
	\begin{lemma} \label{lemm:m_continuous}
		Let Assumption \ref{assum:main} hold true,
		and $(\tau_n)_{n \ge 1}$ be the sequence of $\G$-stopping times therein,
		then for all bounded $\Fc_t$-measurable random variable $\xi$,
		$$
			\E \big[ \xi \1_{\{t \le  \tau_n \}}  \big| \Gc_{\tau_n \wedge t} \big]
			~=~
			\E \big[ \xi \1_{\{t \le  \tau_n \}}  \big| \Gc_{t} \big]
			~=~
			\E \big[ \xi \1_{\{t \le  \tau_n \}}  \big| \Gc \big],
			~\mbox{a.s.}
		$$
		Consequently,
		$$
			m_t = \Lc(X_{\tau_n \wedge t}  | \Gc_t) = \Lc(X_{\tau_n \wedge t}  | \Gc) ,~\mbox{a.s., on}~\{t \le \tau_n \},
			~\mbox{for all}~ t \ge 0~\mbox{and}~n \ge 1.
		$$
		In particular, one can choose $m = (m_t)_{t \ge 0}$ to be a continuous $\Pc_2(\R^d)$-valued process, under the Wasserstein distance $\Wc_2$ on $\Pc_2(\R^d)$.
	\end{lemma}
	\proof
	$\mathrm{(i)}$
	Let us fix $n \ge 1$ and $t \ge 0$.
	Since $\Gc$, $\Gc_t$ and $\Gc_{\tau_n \wedge t}$ are all countably generated,
	let us take respectively a regular conditional probability $(\P_{\om})_{\om \in \Om}$ of $\P$ w.r.t. $\Gc$,
	a regular conditional probability $(\P^t_{\om})_{\om \in \Om}$ of $\P$ w.r.t. $\Gc_t$,
	and a regular conditional probability $(\P^{n,t}_{\om})_{\om \in \Om}$ of $\P$ w.r.t. $\Gc_{\tau_n \wedge t}$.
	Under the $(H)$-hypothesis condition, one has, for all bounded $\Fc_t$-measurable random variables $\xi$,
	$$
			\E \big[ \xi \1_{\{t \le  \tau_n \}}  \big| \Gc_{\tau_n \wedge t} \big]
			~=~
			\E \big[ \xi \1_{\{t \le  \tau_n \}}  \big| \Gc_{t} \big]
			~=~
			\E \big[ \xi \1_{\{t \le  \tau_n \}}  \big| \Gc \big],
			~\mbox{a.s.}
	$$
	Since $\Fc$ is assumed to be countably determined and $\Fc_t \subset \Fc$, this implies that
	$$
		\P^{n,t}_{\om} [B]= \P^t_{\om} [B] = \P_{\om} [B], ~\mbox{for all}~ B \in \Fc_t,~\mbox{for}~\P \mbox{-a.e.}~\om \in \{t \le \tau_n \}.
	$$
	As $\{t \le \tau_n \} \in \Gc_{\tau_n \wedge t}$, one can then deduce that
	$$
		m_t (\om) = \P_{\om} \circ X_t^{-1} =\P^t_{\om} \circ X^{-1}_t = \P^{n,t}_{\om} \circ X_t^{-1} = \P^{n,t}_{\om} \circ X_{\tau_n(\om) \wedge t}^{-1},
		~\mbox{for}~\P \mbox{-a.e.}~\om \in \{t \le \tau_n \}.
	$$
	
	\vspace{0.5em}

	\noindent $\mathrm{(ii)}$
	Given the above, one can assume that the integrability condition in \eqref{eq:SquareIntX} holds for $(X, A, M, M^{\circ})$ in place of $(X_{\tau_n \wedge \cdot}, A_{\tau_n \wedge \cdot}, M_{\tau_n \wedge \cdot}, M_{\tau_n \wedge \cdot}^{\circ})$, up to using a standard localizing technique.  Then,
	$$
		\E^{\P} \left[ \sup_{0 \le s \le t} X_s^2 \right] \le 2X_0^2 + 16\E^{\P} \left[ |M^X_t|^2 \right] + 4\E^{\P} \left[ \big|A\big|_t^2 \right] < \infty, ~\mbox{for all}~ t > 0,
	$$
	 {in which we also used} Doob's inequality. This implies that $\E^{\P_\om} \left[ \sup_{0 \le s \le t} X_s^2 \right] < \infty $, for $\P$- {a.e.} $\om$.
	Define $m_s(\om) := \P_{\om}\circ X_s^{-1}$ for $0 \le s \le t$, then
	$$
		\lim_{\varepsilon \to 0} \Wc_2^2(m_s(\om), m_{s+\varepsilon}(\om) )
		\le
		\lim_{\varepsilon \to 0} \E^{\P_{\om}} \left[ (X_{s+\varepsilon} - X_s)^2 \right]
		= 0, ~\mbox{for}~ \P\mbox{- {a.e.}}~ \om.
	$$
	\qed

	\begin{remark}
	{\rm
	$\mathrm{(i)}$ Let $W$ and $W^{\circ}$ be two independent Brownian motions, and $\sigma = (\sigma_s)_{s \ge 0}$ and $\sigma^{\circ} = (\sigma^{\circ}_s)_{s \ge 0}$ be progressively measurable such that
		$$
			\E \Big[ \int_0^t  \big( |\sigma_s|^2 + |\sigma^{\circ}_s|^2 \big) ds \Big] < \infty,~~\mbox{for all}~t \ge 0.
		$$
		Let us define $M$ and $M^{\circ}$ by
		$$
			M_t := \int_0^t \sigma_s dW_s,
			~~
			M^{\circ}_t := W^{\circ}_t,
			~~
			t \ge 0,
		$$
		together with a continuous process $A$ with square integrable total variation.
		Then, Assumption \ref{assum:main} holds true.
	
	\vspace{0.5em}
	
	\noindent $\mathrm{(ii)}$
		Since $\sigma^{\circ}$ is assumed to be adapted to the filtration $\F$ (rather than the sub-filtration $\G$),
		the form of processes $X$ as in \eqref{eq:structure_X} covers the usual McKean-Vlasov SDEs with common noise:
		$$
			X_t = X_0 + \int_0^t b(X_s, m_s) ds + \int_0^t \sigma(X_s, m_s) dW_s + \int_0^t \sigma_0(X_s, m_s) dW^{\circ}_s,
		$$
		for two independent Brownian motions $W$ and $W^{\circ}$ and $m_s := \Lc(X_s | W^{\circ})$.
		In our general formulation, the process $M^{\circ}$ in \eqref{eq:structure_X} plays the role of the common noise.
	}
	\end{remark}
	
	For sake of more generality, we now also consider a $\R^d$-valued continuous weak Dirichlet process $\left(Y_t\right)_{t \ge 0}$ (see Definition \ref{def:weakDiri}), 
	on the same filtered space $(\Om, \Fc, \F, \P)$, with finite quadratic variation (i.e. $[Y,Y]_t < \infty$ for all $t \ge 0$), whose unique decomposition is given by
	\begin{equation}
		Y_t=Y_0+ A^Y_t + M^Y_t, ~t \ge 0,
	\end{equation}
	for an orthogonal process $A^Y$ and a local martingale  $M^Y$, such that $A^Y_0 = M^Y_0 = 0$.
	Notice that there is no condition on the joint law or dynamics of $X$ and $Y$, 
	it is  just required that $X$ is a semimartingale and $Y$ is a weak Dirichlet process w.r.t. the same filtration $\F$.

	\vspace{0.5em}
	
	For a function $F: \R_+ \x \R^d \x \Pc_2(\R^d) \longrightarrow \R$,
	we denote by $D_m F: \R_+ \x \R^d \x \Pc_2(\R^d) \x \R^d \longrightarrow \R^d$ its partial derivative in the sense that $(m, x) \longmapsto D_m F(t,y, m, x)$ is the derivative of $m \longmapsto F(t,y,m)$ as defined in \eqref{eq:def_DmF}.
	Then, we say  {that}
	\begin{equation} \label{eq:def_C11}
		F \in C^{0,1,1}(\R_{+} \x \R^d \x \Pc_2(\R^d))
	\end{equation}
	if $F$ and its partial derivatives $D_y F$ and $D_m F$ are well-defined and all (jointly) continuous.
	
	\vspace{0.5em}
	
	In the following, for a random variable $\xi$, we  {use the notation}
	$$
		\E^{\circ} \big[\xi \big]  {:=} \E \big[\xi \big| \Gc \big]
	$$
	 {whenever the right-hand side is well-defined, and}
	$$
		\E^{\circ} \big[D_m F(s, \cdot, m_s, X_s) \sigma^{\circ}_s \big] (Y_t) := \E \left[D_m F(s,y,m_s,X_s) \sigma_s^{\circ} \Big| \Gc \right]_{y = Y_t},
		~\mbox{for all}~s, t \ge 0.
	$$

	{We shall require the following local boundedness assumption on $D_{m}F$.}
	\begin{assumption}\label{assum:second}
		With the same sequence $(\tau_n)_{n \ge 1}$ of stopping times as in Assumption \ref{assum:main},
		for all $n \ge 1$, $T>0$ and compact subsets $K \subset \R^d$,
		there exists a constant $C > 0$ satisfying
		\begin{align*}
			\E^{\circ} \Big[
			\left(
			D_m F(r, y, m_s^{n, \lambda,t}, X_s^{n, \eta,t})
			\right)^2
			\Big]
			\le C,~\mbox{a.s.},
			~\mbox{for all}~ & r \in [0, 2T], s \in [0, t],~t \in [0, T], \\
           & ~~~~~~~~~~~~~~~~~~~\lambda, \eta \in [0,1], ~y \in K,
		\end{align*}
		where $m^{n, \lambda, t}_s := m_{\tau_n \wedge s} + \lambda ( m_{\tau_n \wedge t} - m_{\tau_n \wedge s})$ and $X^{n, \eta, t}_s := X_{\tau_n \wedge s} + \eta(X_{\tau_n \wedge t} - X_{\tau_n \wedge s})$.
	\end{assumption}

	\begin{remark}
		If there exists a  constant $C> 0$ such that
		$$
			\big| D_m F(r, y, m, x) \big|
			~\le~
			C(1 + |x|), ~\mbox{for all}~(r, y, m, x) \in \R_+ \x \R^d \x \Pc_2(\R^d) \x \R^d,
		$$
		then  one can check that Assumption \ref{assum:second} holds true  {whenever  Assumption \ref{assum:main} does}.
	\end{remark}

We can now state the main result of this section.
	\begin{theorem}\label{thm:ItoC1}
		Let $F \in C^{0,1,1}(\R_{+} \x \R^d \x \Pc_2(\R^d))$,
		and Assumptions \ref{assum:main} and \ref{assum:second} hold true.
		Then,
		\begin{align}
			F(t, Y_t, m_t)
			=
			F(0,Y_0,m_0)
			& +
			\int_{0}^{t} D_y F(s,Y_s,m_s)  ~d M^Y_s  \nonumber \\
			&+
			\int_{0}^{t}  \E^{\circ} \big[D_m F(s, \cdot, m_s, X_s) \sigma^{\circ}_s \big] (Y_s) dM_s^{\circ}
			+
			\Gamma_t,
			~~
			t \ge 0, \label{eq:ItoC1}
		\end{align}
		where $(\Gamma_t)_{0 \le t \le T}$ is an orthogonal process.
	\end{theorem}
	
	{Before to provide the proof of this result, let us make some comments.}
	
	\begin{remark}
	$\mathrm{(i)}$
		The above theorem proves that $(\F(t, Y_t, m_t))_{t \ge 0}$ is a weak Dirichlet process.
		Moreover, it is  {continuous} so that the decomposition in \eqref{eq:ItoC1} is unique.
		
		\vspace{0.5em}
	
		\noindent $\mathrm{(ii)}$ The above result extends the classical $C^1$-It\^o's formula such as in Gozzi and Russo \cite{Gozzi and Russo}.
		Our new feature is that $F$ depends on $(m_t)_{t \ge 0}$, the conditional marginal distribution of the semimartingale $X$.
		The main reason to consider a semimartingale $X$ rather than a general weak Dirichlet process is that, technically, we will use the integral w.r.t. the finite variation part $A$ of $X$ to handle to conditional expectation terms in the proof
		(see in particular the proof of Lemma \ref{lemm:QV_exp_lim}).
		It is  an open question to us whether this formula  still holds  true if the process $A$ in \eqref{eq:structure_X} is only assumed to be orthogonal, so that $X$ is only a weak Dirichlet process.

	\end{remark}
	
	\begin{proof}[Proof of Theorem \ref{thm:ItoC1}.]
	$\mathrm{(i)}$
	Let $(\tau^Y_n)_{n \ge 1}$ be a sequence of $\F$-stopping times  such that $\tau^Y_n \longrightarrow \infty$ as $n \longrightarrow \infty$, and  $Y$ is uniformly bounded on $[0, \tau^Y_n]$, for each $n \ge 1$.
	Then, given the sequence $(\tau_n)_{n \ge 1}$ of localizing stopping times given in Assumption \ref{assum:main}, we define
	$$
		X^n_t := X_{\tau_n \wedge t},
		~~
		Y^n_t := Y_{\tau^Y_n \wedge t},
		~~\mbox{and}~~
		m^n_t := \Lc(X^n_t | \Gc_t),
		~~
		t \ge 0, ~~ n \ge 1.
	$$
	It is enough to prove that \eqref{eq:ItoC1}  holds for $(X^n, Y^n, m^n)$, and then use Lemma \ref{lemm:m_continuous} and let $n \longrightarrow \infty$.
	For simplicity, we also assume that  $F(0,Y_0,m_0) = 0$ and $d=1$.
	Upon replacing the processes $(X, Y, m)$ in \eqref{eq:ItoC1} by the localized process $(X^n, Y^n, m^n)$,
	 one can assume w.l.o.g.~that
\begin{align}\label{eq: new ass for local}
\begin{array}{c} \mbox{$F(0,Y_0,m_0) = 0$, $d=1$,} \\
\mbox{$Y$ is bounded,}\\
\mbox{Assumption \ref{assum:main} and \ref{assum:second} hold with $\tau_{n} \equiv \infty$ a.s.~for all $n\ge 1$}
\end{array}
\end{align} which we do in the following.

	\vspace{0.5em}

	Let us define the process $(\Gamma_t)_{t \ge 0}$ by
	\begin{align*}
		\Gamma_{t} ~:=~ F(t,Y_t,m_t)
		- \int_{0}^{t} D_y F(s,Y_s,m_s) dM^Y_s
		- \int_{0}^{t}  \E^{\circ} \big[D_m F(s, \cdot, m_s, X_s) \sigma^{\circ}_s \big] (Y_s) dM_s^{\circ}.
	\end{align*}
	To prove the theorem, it is enough to show that $\Gamma$ is an orthogonal process, i.e. $\left[\Gamma,N \right] = 0$ for any  {bounded} continuous  {martingale} $N$.
	 {That is, for $N$ given:}
	$$
		I^{\varepsilon}_t := \frac{1}{\varepsilon}
		\int_{0}^{t}  \big[
			F(s+\varepsilon, Y_{s+\varepsilon},m_{s+\varepsilon})
			- F(s,Y_s,m_s)
			\big] (N_{s+\varepsilon} - N_s)ds
			~\longrightarrow~
			I_t, ~t \ge 0, ~\mbox{u.c.p.},
	$$
	as $\eps \longrightarrow 0$, where
	\begin{align*}
		I_{t} :=  \int_{0}^{t} D_y F(s,Y_s,m_s) ~d[M^Y,N ]_s
		+ \int_{0}^{t} \E^{\circ} \Big[D_m F(s,\cdot,m_s,X_s)
		\sigma_s^{\circ} \Big] (Y_s)~ d[M^{\circ},N]_s,
		~t \ge 0.
	\end{align*}
	To this end, we use the decomposition
	$$
		I^{\varepsilon}_t = I^{1,\varepsilon}_t + I^{2,\varepsilon}_t + I^{3, \eps}_t, ~t \ge 0,
	$$
	with
	\begin{align}\label{eq: def I1}
		I_t^{1,\varepsilon}
		~:=~
		\int_{0}^{t}
		\big[
			F(s+\varepsilon,Y_{s+\varepsilon},m_{s+\varepsilon})
			-
			F(s+\varepsilon,Y_s,m_{s+\varepsilon})
		\big]
		\frac{N_{s+\varepsilon}-N_s}{\varepsilon} ds,
	 \end{align}
	\begin{align}\label{eq: def I2}
		I_t^{2,\varepsilon}
		~:=~
		\int_{0}^{t}
		\big[
			F(s+\varepsilon,Y_s,m_{s+\varepsilon})
			-
			F(s+\varepsilon,Y_s,m_{s})
		\big]
		\frac{N_{s+\varepsilon}-N_s}{\varepsilon} ds,
	 \end{align}
	and
	\begin{align}\label{eq: def I3}
		I^{3,\eps}_t
		~:=~
		\int_{0}^{t}
		\big[
			F(s+\varepsilon,Y_s,m_{s})
			-
			F(s,Y_s,m_{s})
		\big]
		\frac{N_{s+\varepsilon}-N_s}{\varepsilon} ds.
	\end{align}	
	We shall prove in Lemmas \ref{lemm:cvg_I1} and \ref{lemm:cvg_I2} below that
	$$
		I^{1,\eps}_t \longrightarrow \int_{0}^{t} D_y F(s,Y_s,m_s) d[M^Y,N]_s,
		~\mbox{and}~
		I^{3,\eps}_t \longrightarrow 0,
		~t \ge 0,
		~\mbox{u.c.p.},
	$$
	and,  {in} Lemma \ref{THM1}, that
	$$
		I_t^{2,\varepsilon}
		~\longrightarrow~
		\int_{0}^{t}  \E^{\circ} \big[D_m F(s, \cdot, m_s, X_s) \sigma^{\circ}_s \big] (Y_s) d[M^{\circ},N]_s,
		~t \ge 0, \mbox{u.c.p.}
	$$
	This  {will provide the required result.}
	\end{proof}

	\begin{lemma} \label{lemm:cvg_I1}
		 Let the conditions of Theorem \ref{thm:ItoC1} and Condition \eqref{eq: new ass for local} hold. Let $(I^{1,\varepsilon})_{\varepsilon>0}$ be defined as in \eqref{eq: def I1}. Then,
		$$
			I_t^{1,\varepsilon}  \longrightarrow \int_{0}^{t} D_y F(s,Y_s,m_s) d[M^Y,N]_s,~t \ge 0, ~\mbox{u.c.p., as}~ \eps \longrightarrow 0.
		$$
	\end{lemma}
	\begin{proof}
	Let us define
	$$
		I_t^{11,\varepsilon}
		:=
		\int_{0}^{t} D_y F(s,Y_s,m_{s}) (Y_{s+\varepsilon} - Y_s)
			\frac{N_{s+\varepsilon}-N_s}{\varepsilon}ds, ~t \ge 0,
	$$
	and
	$$
		I_t^{12,\varepsilon}
		:=
		\int_{0}^{t} \Delta_s^{\varepsilon} (Y_{s+\varepsilon} -Y_s)
			\frac{N_{s+\varepsilon}-N_s}{\varepsilon} ds, ~t \ge 0,
	$$
	with
	$$
		\Delta_s^{\varepsilon}
		:=
		\int_{0}^{1} \left(D_y F(s+\varepsilon, Y_s+\lambda(Y_{s+\varepsilon}-Y_s),m_{s+\varepsilon})
			-D_y F(s,Y_s,m_s)
		\right)d\lambda,
	$$
	so that  $I^{1,\varepsilon}_t = I^{11,\varepsilon}_t + I^{12,\varepsilon}_t$ for all $t \ge 0$.

	\vspace{0.5em}

	First, by similar arguments as in \cite[Proposition 3.10]{Gozzi and Russo}, one easily obtains that
	$$
		I_t^{11,\varepsilon} \xrightarrow[\varepsilon \to 0]{\text{u.c.p.}}
		\int_{0}^{t} D_y F(s,Y_s,m_s) d\left[M^Y,N\right]_s.
	$$
	Further, by (uniform) continuity of $(s,y) \longmapsto D_y F(s, y, m_s)$ on compact sets,
	there exists random variables  {$(\delta(D_y F, \eps))_{\eps>0}$} such that
	$$
		\sup_{0 \le s \le t} \big|\Delta_s^{\varepsilon} \big|
		~\le~
		\delta(D_y F, \eps)
		~\longrightarrow~
		0,
		~\mbox{a.s., as}~
		\eps \longrightarrow 0.
	$$
	Recall that $Y$ has finite quadratic variation and  {that} $N$ is square integrable, so that
	$$
		\left(\int_{0}^{t}\frac{(Y_{s+\varepsilon}-Y_s)^2}{\varepsilon}ds \right) \left(\int_{0}^{t}\frac{(N_{s+\varepsilon}-N_s)^2}{\varepsilon}ds \right)
		\xrightarrow[\varepsilon \to 0]{\text{u.c.p.}} [Y]_t [N]_t < \infty.
	$$
	It follows  {that}
	$$
		|I_t^{12,\varepsilon}|
		~\le~
		\delta\left(D_y F, \varepsilon \right)
		\sqrt{\int_{0}^{t}\frac{(Y_{s+\varepsilon}-Y_s)^2}{\varepsilon}ds \int_{0}^{t}\frac{(N_{s+\varepsilon}-N_s)^2}{\varepsilon}ds}
		~\longrightarrow~
		0, ~t \ge 0,
		~\mbox{u.c.p.}
	$$
	as $\eps\longrightarrow 0$.
	\end{proof}

	\begin{lemma} \label{lemm:cvg_I2}
		 Let the conditions of Theorem \ref{thm:ItoC1} and Condition \eqref{eq: new ass for local} hold. Let $(I^{3,\varepsilon})_{\varepsilon>0}$ be defined as in \eqref{eq: def I3}. Then,  $I^{3,\eps}_t \longrightarrow 0$, $t \ge 0$, u.c.p.~{as $\eps\to 0$.}
	\end{lemma}
	\begin{proof}
		By the integration by parts formula, one can rewrite $I^{3,\eps}_t$ as
		$$
			I_t^{3,\varepsilon} := \int_{0}^{t+\eps} \eta^{\varepsilon}_r dN_r,
			~\mbox{with}~
			\eta^{\varepsilon}_r
			:=
			\frac{1}{\eps} \int_{(r-\varepsilon)_+}^{r \wedge t}
			\big( F(s+\varepsilon,Y_{s},m_{s}) - F(s,Y_s, m_{s}) \big) ds.
		$$
		One observes that $\eta^{\eps}_r \longrightarrow 0$ as $\eps \longrightarrow 0$.
		Moreover, $|\eta^{\eps}|$ is bounded by the  (locally bounded) continuous adapted process $(\bar \eta_r)_{r \ge 0}$ defined as:
		$$
			\bar \eta_r := {2}\max_{s \le r+1}\max_{r' \le r}  \big| F(s, Y_{r'}, m_{r'})  \big|.
		$$
		Then, one can apply e.g. \cite[Theorem I.4.31]{Jacod Shiryaev}  {to deduce that}
		$$
			I_t^{3,\varepsilon} \longrightarrow 0, ~t \ge 0, ~\mbox{u.c.p.}
		$$
	\end{proof}

	To prove Lemma \ref{THM1} below, we need the following two  {intermediate} lemmas.
	\begin{lemma}\label{lemm:conditional_prpty}
		 Let the conditions of Theorem \ref{thm:ItoC1} and Condition \eqref{eq: new ass for local} hold. Let $H$ be a $\F$-progressively measurable process such that
		$$
			\Mb_t := \int_0^t H_s d M_s, ~t \ge 0,
			~\mbox{is a martingale}.
		$$
		Then,
		$$
			\E^{\circ} \big[ \Mb_t \big] = 0, ~\mbox{a.s., for all}~t \ge 0.
		$$
		 {Moreover,}
		$$
			\E^{\circ} \left[ \left( M_t - M_s \right)^2 \right]
			=
			\E^{\circ} \Big[ [ M ]_t - [ M ]_s  \Big],
			~\mbox{a.s., for all}~ t \ge s \ge 0.
		$$
	\end{lemma}
	\begin{proof}
	\noindent $\mathrm{(i)}$ Recall that $\G$ is the filtration generated by $M^{\circ}$.
	Then, for all $\phi \in C_b(C^d[0,T];\R^d)$, there exists a bounded $\G$-martingale  $ {\Phi} = ( {\Phi}_r)_{r \ge 0}$
	such that $ {\Phi}_r = \E \big[ \phi( {M^{\circ}}) \big| \Gc_r \big]$, for all $r \ge 0$.
	By Assumption \ref{assum:main}.$\mathrm{(ii)}$, one knows that $[\Mb,  {\Phi}]_r = 0$ for all $r \ge 0$.
	Up to a localization argument, one obtains that
	\begin{align*}
		\E \big[\Mb_t \phi( {M^{\circ}}) \big]
		~=~
		\E \left[\Mb_0  {\Phi}_0 + \int_0^t \Mb_r d {\Phi}_r + \int_0^t  {\Phi}_r d \Mb_r + [\Mb,  {\Phi} ]_t \right]
		~ = ~
		0.
	\end{align*}
	Using the (H)-hypothesis condition in Assumption \ref{assum:main}, it follows that
	$$
		\E^{\circ}[\Mb_t]  ~:=~ \E \big[ \Mb_t \big| \Gc \big] = \E \big[ \Mb_t \big| \Gc_t \big] = 0.
	$$
	
	\vspace{0.5em}
	
	\noindent $\mathrm{(ii)}$  {Given} the square integrability conditions on $M$  of  Condition \eqref{eq: new ass for local}, both processes
	$$
		\int_0^t M_r dM_r = \frac12 M^2_t - \frac12 [M]_t,
		~~
		\int_{0}^{t} M_s \mathbf{1}_{\{ r \ge s\}} dM_r,~~t \ge 0,
	$$
	are true martingales.
	Thus
	$$
		\E^{\circ} \Big[ \int_s^t (M_r - M_s) dM_r \Big] = 0.
	$$
	Then, it follows that
	\begin{align*}
		\E^{\circ} \left[\left(M_t - M_s \right)^2 \right]
		~=~
		\E^{\circ} \Big[ \int_s^t 2 (M_r - M_s) dM_r  + [ M ]_t - [M]_s  \Big]
		~=~
		 \E^{\circ} \Big[ [ M ]_t - [M]_s  \Big].
	\end{align*}
	\end{proof}

	\begin{lemma}\label{lemm:QV_exp_lim}
		Let the conditions of Theorem \ref{thm:ItoC1} and Condition \eqref{eq: new ass for local} hold. Then,
		$$
			\lim_{\varepsilon \to 0} \E^{\circ} \left[ \frac{1}{\varepsilon} \int_0^t \left(|A|_{s+\varepsilon} - |A|_s \right)^2 ds \right]
			=
			0,
			~t \ge 0,~\mbox{u.c.p.},
		$$
		\begin{equation*}
			\lim_{\varepsilon \to 0} \E^{\circ} \left[ \frac{1}{\varepsilon} \int_0^t \left(M_{s+\varepsilon} - M_s \right)^2 ds \right]
			=
			\E^{\circ} \big[ [M]_t \big],
			~t \ge 0,~\mbox{u.c.p.}
		\end{equation*}
		and
		$$
		\lim_{\varepsilon \to 0} \E^{\circ} \left[ \frac{1}{\varepsilon} \int_0^t \left( \int_{s}^{s+\varepsilon} \sigma^{\circ}_r dM^{\circ}_r \right)^2 ds \right]
		~=~
		\int_0^t \E^{\circ} \big[ (\sigma^{\circ}_s)^2\big] d[M^{\circ}]_s ,
		~t \ge 0, ~\mbox{u.c.p.}
		$$
	\end{lemma}
	\begin{proof}
	$\mathrm{(i)}$ Notice that the total variation process $(|A|_t)_{t \ge 0}$ of $A$ is a continuous non-decreasing process.
	By direct computations, it follows that
	\begin{align*}
		& ~ \lim_{\varepsilon \to 0}
		\E^{\circ} \left[ \frac{1}{\varepsilon}\int_{0}^{t}
		\left(\left|A\right|_{s+\varepsilon} - \left|A\right|_s \right)^2 ds
		\right]  \\
		= &~
		\lim_{\varepsilon \to 0}
		\frac{1}{\eps}
		\Big[
		\int_t^{t+\eps} \E^{\circ} \big[ |A|_s^2 \big] ds
		-
		\int_0^{\eps} \E^{\circ} \big[ |A|_s^2 \big] ds
		-
		2 \int_{0}^t \E^{\circ} \big[  |A|_s \big( |A|_{s+\varepsilon} - |A|_s \big) \big] ds
		\Big]
		\\
		= &~
		\E^{\circ} \big[ |A|_t^2 \big]
		-
		\lim_{\varepsilon \to 0}
		\E^{\circ} \left[ 2 \int_{0}^{t}
		\frac{1}{\varepsilon}
		\int_{s}^{s+\varepsilon} \left|A\right|_s d\left|A\right|_r ds
		\right] \\
		= &~
		\E^{\circ} \big[ |A|_t^2 \big]
		-
		\lim_{\varepsilon \to 0}
		\E^{\circ} \left[ 2 \int_{0}^{t+\eps}
		\frac{1}{\varepsilon}
		\int_{(r-\varepsilon)_+}^{r \wedge t}
			\left|A\right|_s ds d\left|A\right|_r
		\right] \\
		= & ~
		\E^{\circ} \big[ |A |_t^2  \big]
		-
		\E^{\circ} \left[\int_{0}^{t}
                   2 \left|A\right|_r
                   d\left|A\right|_r
                   \right]\
                   =~ 0, ~\mbox{a.s.}
	\end{align*}
	In the above, we used the square integrability condition on $|A|_t$ in Condition \eqref{eq: new ass for local}  to deduce that
	\begin{align*}
		\lim_{\varepsilon \to 0}
		\E^{\circ} \left[ \int_{0}^{t+\eps}
		\frac{1}{\varepsilon}
		\int_{(r-\varepsilon)_+}^{r \wedge t}
			\left|A\right|_s ds d\left|A\right|_r
		\right]
		&=
		\E^{\circ} \left[ \lim_{\varepsilon \to 0} \int_{0}^{t+\eps}
		\frac{1}{\varepsilon}
		\int_{(r-\varepsilon)_+}^{r \wedge t}
			\left|A\right|_s ds d\left|A\right|_r
		\right] \\
		&=
		\E^{\circ} \left[\int_{0}^{t}
                   \left|A\right|_r
                   d\left|A\right|_r
                   \right], ~\mbox{a.s.}
	\end{align*}

	\vspace{0.5em}
	
	\noindent $\mathrm{(ii)}$ Next, by Lemma \ref{lemm:conditional_prpty}, it follows that
	$$
		\lim_{\varepsilon \to 0} \E^{\circ} \left[ \frac{1}{\varepsilon} \int_0^t \left(M_{s+\varepsilon} - M_s \right)^2 ds \right]
		=
		\lim_{\eps \to 0} \frac{1}{\eps} \int_0^t \E^{\circ} \big[ [M]_{s+\eps} - [M]_s \big] ds
		=
		\E^{\circ} \big[ [ M]_t \big], ~\mbox{a.s.}
	$$

	\noindent $\mathrm{(iii)}$ Let us set
	$$
		\Mb^{\circ}_t := \int_0^t \sigma^{\circ}_s dM^{\circ}_s, ~t \ge 0.
	$$
	Then,
	\begin{align} \label{eq:V2_eps_M0}
		& \E^{\circ} \Big[ \frac{1}{\varepsilon} \int_{0}^{t}   \big( \Mb^{\circ}_{s+\varepsilon} - \Mb^{\circ}_s \big)^2 ds \Big]  \nonumber \\
		=&~
		\E^{\circ}  \Big[
			\int_{0}^{t} \frac{1}{\eps} \Big( \int_s^{s+\eps} 2 (\Mb^{\circ}_r - \Mb^{\circ}_s) d \Mb^{\circ}_r + \int_s^{s+\eps} d [\Mb^{\circ}]_r \Big) ds
		\Big] \nonumber \\
		= &~
		2 \int_0^{t+\eps} \frac{1}{\eps} \int_{(r-\eps)_+}^{r \wedge t} \E^{\circ}\big[ (\Mb^{\circ}_r - \Mb^{\circ}_s) \sigma^{\circ}_r \big] ds ~d M^{\circ}_r
		~+~
		\E^{\circ} \Big[ \int_0^{t+\eps} \frac{r \wedge t - (r-\eps)_+ }{\eps}  d[\Mb^{\circ}]_r \Big].
	\end{align}
	Let us define
	$$
		H^{\eps}_r := \frac{1}{\eps} \int_{(r-\eps)_+}^{r \wedge t} \E^{\circ}\big[ (\Mb^{\circ}_r - \Mb^{\circ}_s) \sigma^{\circ}_r \big] ds.
	$$	
   Under the square integrability conditions in  \eqref{eq: new ass for local}, it is easy to check that
	\begin{equation} \label{eq:Heps0}
		\big| H^{\eps}_r \big|
		~\le~
		\sqrt{ \E^{\circ} \big[ |\sigma^{\circ}_r|^2 \big]}
		~
		\frac{1}{\eps} \int_{(r-\eps)_+}^{r\wedge t} \sqrt{ \E^{\circ} \big[ \big(\Mb^{\circ}_r - \Mb^{\circ}_s \big)^2 \big]} ds
		~\longrightarrow~
		0, ~\mbox{as}~ \eps \longrightarrow 0.
	\end{equation}
	Moreover, one has
	$$
		\big| H^{\eps}_r \big|
		~\le~
		H_r
		~:=~
		2 \sqrt{ \E^{\circ} \big[ |\sigma^{\circ}_r|^2 \big]} ~ \sup_{s \le r}  \sqrt{ \E^{\circ} \big[ (\Mb^{\circ}_s)^2 \big] }.
	$$
	Again, under the square integrability conditions in \eqref{eq: new ass for local}, the process $\big(\E^{\circ} \big[ (\Mb^{\circ}_s)^2 \big] \big)_{s \ge 0}$ is continuous so that one can localize it by a sequence of stopping times, which can be considered to be $(\tau_n)_{n \ge 1}$  w.l.o.g. It follows that, for a sequence of positive constants $(C_n)_{ n \ge 1}$,
	$$
		\E \Big[ \int_0^{t \wedge \tau_n} H^2_r d[M^{\circ}]_r \Big]
		~\le~
		C_n \E \Big[ \int_0^t |\sigma^{\circ}_r|^2 d[M^{\circ}]_r \Big] < \infty.
	$$
	Together with \eqref{eq:Heps0}, one can deduce that
	\begin{equation} \label{eq:cvg_int_H}
		\int_0^{t+\eps} \frac{1}{\eps} \int_{(r-\eps)_+}^{r \wedge t} \E^{\circ}\big[ (\Mb^{\circ}_r - \Mb^{\circ}_s) \sigma^{\circ}_r \big] ds ~d M^{\circ}_r
		=
		\int_0^{t+\eps} H^{\eps}_r d M^{\circ}_r
		\longrightarrow
		0, ~t \ge 0,
		~\mbox{u.c.p.}
	\end{equation}
	Further, it is easy to see that
	$$
		\E^{\circ} \Big[ \int_0^{t+\eps} \frac{r \wedge t - (r-\eps)_+ }{\eps}  d[\Mb^{\circ}]_r \Big]
		~\longrightarrow~
		\E^{\circ} \big[ [\Mb^{\circ}]_t \big], ~\mbox{a.s.}
	$$
	Together with \eqref{eq:V2_eps_M0} and \eqref{eq:cvg_int_H}, this leads to
	$$
		\lim_{\varepsilon \to 0} \E^{\circ} \left[ \frac{1}{\varepsilon} \int_0^t \left( \Mb^{\circ}_{s+\varepsilon} - \Mb^{\circ}_s \right)^2 ds \right]
		~=~
		\E^{\circ} \big[ [\Mb^{\circ}]_t \big]
		=
		\int_0^t \E^{\circ} \big[ (\sigma^{\circ}_s)^2\big] d[M^{\circ}]_s ,
		~t \ge 0, ~\mbox{u.c.p.}
	$$
	This concludes the proof.
	\end{proof}

	\begin{lemma}\label{THM1}
		 Let the conditions of Theorem \ref{thm:ItoC1} and Condition \eqref{eq: new ass for local} hold. Let $(I^{2,\varepsilon})_{\varepsilon>0}$ be defined as in \eqref{eq: def I2}. Then,
		$$
			I_t^{2,\varepsilon}
			~\longrightarrow~
			\int_{0}^{t}  \E^{\circ} \big[D_m F(s, \cdot, m_s, X_s) \sigma^{\circ}_s \big] (Y_s) d[M^{\circ},N]_s,
			~t \ge 0, \mbox{u.c.p., as}~
			\eps \longrightarrow 0.
		$$
	\end{lemma}

	\begin{proof}
	Recall that
	$$
		I_t^{2,\varepsilon}
		~:=~
		\int_{0}^{t}
		\big[
			F(s+\varepsilon,Y_s,m_{s+\varepsilon})
			-
			F(s+\varepsilon,Y_s,m_{s})
		\big]
		\frac{N_{s+\varepsilon}-N_s}{\varepsilon} ds,
	$$
	and that
	$$
		m_s = \Lc(X_s | \Gc), ~s \ge 0.
	$$
	By the definition of $\delta_m F$ and $D_m F$ in \eqref{eq:def_dmF} and \eqref{eq:def_DmF},
	it follows that
	\begin{align*}
		I^{2, \eps}_t
		= &
		\int_0^t \int_0^1
		\E^{\circ} \Big[ \delta_m F(s+\varepsilon, \cdot, m_s^{\lambda,\varepsilon}, X_{s+\varepsilon})-\delta_m F(s+\varepsilon, \cdot, m_s^{\lambda,\varepsilon}, X_s) \Big] (Y_s)
		\frac{N_{s+\varepsilon}-N_s}{\varepsilon} d\lambda ds \\
		= &
		\int_0^t \int_0^1 \int_0^1
		\E^{\circ} \Big[ D_m F(s+\varepsilon, \cdot, m_s^{\lambda, \varepsilon}, X_s^{\eta,\varepsilon}) (X_{s+ \varepsilon}-X_s) \Big] (Y_s)
		\frac{N_{s+\varepsilon}-N_s}{\varepsilon} d\eta  d\lambda ds,
	\end{align*}
	where
	$$
		m^{\lambda, \eps}_s := m_s + \lambda ( m_{s+\eps} - m_s)
		~\mbox{and}~
		X^{\eta, \eps}_s := X_s + \eta(X_{s +\eps} - X_s).
	$$
	Let us write
	$$
		I^{2, \eps}_t
		~=~
		J_t^{1,\varepsilon} + J_t^{2,\varepsilon}, ~t \ge 0,
	$$
	where
	$$
		J_{t}^{1,\varepsilon} :=
		\int_0^{t} \int_0^1 \int_0^1 \E^{\circ}
			\Big[
				\Delta_m F(s, s+\eps, \cdot, \lambda, \eta)
				~(X_{s+ \varepsilon}-X_s)
			\Big] (Y_s)
		\frac{N_{s+\varepsilon}-N_s}{\varepsilon} d\eta d\lambda  ds,
	$$
	with
	$$
		\Delta_m F(s, t, y, \lambda, \eta)
		:=
		D_m F\big(t, y, m_s+ \lambda(m_t-m_s), X_s + \eta (X_t - X_s) \big)-D_m F \big(t, y, m_s,X_s \big),
	$$
	and
	$$
		J_{t}^{2,\varepsilon} :=
		\int_0^{t} \E^{\circ} \Big[ D_m F(s+\eps, \cdot, m_s, X_s) (X_{s+\varepsilon}-X_s) \Big] (Y_s)
		\frac{N_{s+\varepsilon}-N_s}{\varepsilon}ds.
	$$

	\noindent $\mathrm{(i)}$ To study the limit of $J^{1,\eps}_t$, we notice that the map $(s,t,y) \longmapsto \Delta_m F(s,t,y, \lambda, \eta)$ is continuous.  {Hence, by Condition \eqref{eq: new ass for local},}
	$$
		(s,t,y) \longmapsto \int_0^1 \int_0^1 \E^{\circ} \Big[ \big( \Delta_m F(s, t, y, \lambda, \eta) \big)^2 \Big] d \eta d \lambda
	$$
	is also continuous and hence uniformly continuous on compact sets.
	In particular, one has
	$$
		\lim_{\eps \to 0} \Delta_t(\eps) = 0,
		~\mbox{a.s., with}~
		\Delta_t(\eps) := \sup_{s \le t} \int_0^1 \int_0^1 \E^{\circ} \Big[ \big( \Delta_m F(s, s + \eps, \cdot, \lambda, \eta) \big)^2 \Big] (Y_s) d \eta d \lambda  .
	$$
	Further, by Cauchy-Schwarz inequality,
	\begin{align*}
		\sup_{s \le t} \big |J^{1,\eps}_s \big|
		&~\le~
		\sqrt{\Delta_t(\eps)}
		~\sqrt{ \frac{1}{\varepsilon}\int_{0}^{t}\E^{\circ} \Big[ \big(X_{s+\varepsilon}-X_s \big)^2 \Big] ds}
		~\sqrt{\frac{1}{\eps} \int_0^t (N_{s+\varepsilon}-N_s)^2 ds}.
       \end{align*}
	 {Since}  the limit in probability of
	$$
		\frac{1}{\eps}  \int_0^t (N_{s+\varepsilon}-N_s)^2 ds
		~\mbox{and}~
		\frac{1}{\eps}\int_0^t \E^{\circ} \Big[(X_{s+\varepsilon}-X_s)^2 \Big] ds
	$$
	are both finite a.s., by the fact that $N$ has finite quadratic variation and  by Lemma \ref{lemm:QV_exp_lim} for the right-hand side term,
	 {w}e hence conclude that
	$$
		J^{1,\eps}_t \longrightarrow 0,
		~t \ge 0,
		~\mbox{u.c.p. as}~
		\eps \longrightarrow 0.
	$$

	\noindent $\mathrm{(ii)}$ We next consider $J^{2, \eps}_t$, and write it  {as}
	$$
		J^{2, \eps}_t
		~=~
		J^{21, \eps}_t + J^{22, \eps}_t + J^{23, \eps}_t,
	$$
	where
	$$
		J_{t}^{21,\varepsilon} :=
		\int_0^{t} \E^{\circ} \big[ D_m F(s+\eps, \cdot, m_s, X_s) (A_{s+\varepsilon} - A_s) \big] (Y_s)
		\frac{N_{s+\varepsilon}-N_s}{\varepsilon}ds,
	$$
	$$
		J_{t}^{22,\varepsilon} :=
		\int_0^{t} \E^{\circ} \big[ D_m F(s+\eps, \cdot, m_s, X_s) (M_{s+\varepsilon} - M_s) \big] (Y_s)
		\frac{N_{s+\varepsilon}-N_s}{\varepsilon}ds,
	$$
	and
	$$
		J_{t}^{23,\varepsilon} :=
		\int_0^{t} \E^{\circ} \Big[ D_m F(s+\eps, \cdot, m_s, X_s) \int_s^{s+\eps} \sigma^{\circ}_r dM^{\circ}_r \Big] (Y_s)
		\frac{N_{s+\varepsilon}-N_s}{\varepsilon}ds.
	$$
	First, observe that, under  {Condition \eqref{eq: new ass for local}}, one has
	$$
		\E^{\circ} \Big[
		\left(
		D_m F\left(r, y, m_s^{\lambda,t}, X_s^{\eta,t}\right)
		\right)^2
		\Big] \le C,
	$$
	for some constant $C > 0$, and
	$$
		\big| J^{21,\eps}_t \big|
		~\le~
		\sqrt{C} \sqrt{ \frac{1}{\eps} \int_0^t \E^{\circ} [ (A_{s+\eps}-A_s)^2] ds} ~ \sqrt{ \frac{1}{\eps} \int_0^t (N_{s+\eps} - N_s)^2 ds }
		\longrightarrow 0, ~\mbox{u.c.p.},
	$$
	by Lemma \ref{lemm:QV_exp_lim}.
	
	\vspace{0.5em}

	Next, by Lemma \ref{lemm:conditional_prpty},
	$$
		\E^{\circ} \Big[ D_m F(s+\eps, y, m_s, X_s) \int_s^{s+\eps} dM_r \Big] = 0, ~\mbox{a.s. for all}~s \ge 0.
	$$
	Thus,
	$$
		J^{22,\eps}_t = 0, ~\mbox{for all}~t \ge 0, ~\mbox{a.s.}
	$$
	
	Finally, we study the limit of $J^{23,\eps}_t$.
	Define
   $$
   \psi(s,r,\eps) := \E^{\circ} \Big[D_m F(s+\eps,\cdot,m_s,X_s) \sigma^{\circ}_r \Big](Y_s) -
                       \E^{\circ} \Big[D_m F(r,\cdot,m_r,X_r) \sigma^{\circ}_r \Big](Y_r).
   $$
   We claim that
   \begin{align}\label{eq:claim psi eps}
   \lim_{\eps \to 0} \int_0^{t} \int_s^{s+\eps} \psi(s,r,\eps)  dM^{\circ}_r
		        \frac{N_{s+\varepsilon}-N_s}{\varepsilon}ds
   = 0,~ t \ge 0, ~u.c.p., ~\mbox{as}~ \eps \longrightarrow 0.
  \end{align}
   Then, setting $D_t :=  \int_0^{t}\E^{\circ} \Big[ D_m F(r, \cdot, m_r, X_r)
\sigma^{\circ}_r \Big] (Y_r) dM^{\circ}_r$, we can apply \cite[Proposition 1.1]{Russo and Vallois2} to deduce that
	\begin{align*}
		\lim_{\epsilon \to 0} J^{23,\epsilon}_t
		~=~ &       \lim_{\epsilon \to 0} \int_0^{t}  \int_s^{s+\epsilon}\E^{\circ} \Big[ D_m F(r, \cdot, m_r, X_r)
					\sigma^{\circ}_r \Big] (Y_r) dM^{\circ}_r
					\frac{N_{s+\varepsilon}-N_s}{\varepsilon}ds \\
		~=~ &		\lim_{\epsilon \to 0} \int_0^{t}  (D_{s+\epsilon} - D_s)
		\frac{N_{s+\varepsilon}-N_s}{\varepsilon}ds\\
		~=~& [D,N]_t \\
		~=~ &       \int_0^{t} \E^{\circ} \Big[  D_m F(r, \cdot, m_r, X_r)
					{\sigma^{\circ}_r}  \Big] (Y_r)
					d[M^{\circ},N]_r,
		\end{align*}
where the last equality follows from the property of the classical covariation.

	\vspace{0.5em}

   To conclude, it remains to prove our claim \eqref{eq:claim psi eps}.
   By Cauchy-Schwarz inequality,
   \begin{align*}
         &      {\left| J^{23,\eps}_t
               - \int_0^{t} \E^{\circ} \Big[ \int_s^{s+\eps} D_m F(r, \cdot, m_r, X_r)  \sigma^{\circ}_r dM^{\circ}_r \Big] (Y_r)
		        \frac{N_{s+\varepsilon}-N_s}{\varepsilon}ds  \right|}\\
   ~=~   &     {\left| \int_0^{t} \int_s^{s+\eps} \psi(s,r,\eps)  dM^{\circ}_r
		        \frac{N_{s+\varepsilon}-N_s}{\varepsilon}ds \right|}\\
   ~\le~ &     \sqrt{\frac{1}{\eps}\int_0^{t} \left(\int_s^{s+\eps} \psi(s,r,\eps)
               dM^{\circ}_r \right)^2 ds}
		        \sqrt{\frac{1}{\eps}\int_{0}^{t}(N_{s+\varepsilon}-N_s)^2 ds}.
   \end{align*}
	 {Thus, it} suffices to show that
	\begin{equation} \label{eq:interm_limit}
		\frac{1}{\eps}\int_0^{t} \E [(\int_s^{s+\eps} \psi(s,r,\eps)
               dM^{\circ}_r )^2] ds \longrightarrow 0,
               ~\mbox{as}~\eps \longrightarrow 0.
	\end{equation}
	By Ito's isometry,
	\begin{align*}
		\frac{1}{\eps}\int_0^{t} \E \Big[\left(\int_s^{s+\eps} \psi(s,r,\eps)
               dM^{\circ}_r \right)^2 \Big] ds \nonumber
		~=~   &
		\E \Big[ \int_{0}^{t} \frac{1}{\eps} \int_s^{s+\eps}
               \psi(s,r,\eps)^2
               d[M^{\circ}]_r ds
               \Big] \nonumber\\
		~=~   &     \E \Big[ \int_{0}^{t+\eps} \frac{1}{\eps} \int_{(r-\eps)_+}^{r}
               \psi(s,r,\eps)^2 ds d[M^{\circ}]_r
               \Big].
	\end{align*}
	 {Also,} $\frac{1}{\eps} \int_{(r-\eps)_+}^{r}\psi(s,r,\eps)^2 ds \longrightarrow 0$ as $\eps \longrightarrow 0$,
	and
	\begin{align*}
		\frac{1}{\eps} \int_{(r-\eps)_+}^{r}
               \psi(s,r,\eps)^2 ds
		&
		~\le~
		4 C \E^{\circ}\big[(\sigma_r^{\circ})^2\big],
	\end{align*}
	and
	$$
		\E \Big[ \int_{0}^{2t} 4C \E^{\circ}\big[(\sigma_r^{\circ})^2\big] d[M^{\circ}]_r \Big] < +\infty,
	$$
	both thanks to Condition \eqref{eq: new ass for local}.
	Therefore, it follows from the dominated convergence theorem that \eqref{eq:interm_limit} holds, and we hence conclude the proof.
	\end{proof}

\section{A verification theorem for a class of McKean-Vlasov optimal control problems}
\label{sec:Verification}

	Let $\Om^0 = \Om^1 := \Cc([0,T], \R^d)$ be the canonical spaces of $\R^d$-valued continuous paths on $[0,T]$,
	where the canonical process on $\Om^0$ is denoted by $X^0$, and the one on $\Om^1$ is denoted by $W$.
	Under the uniform convergence topology, we define $\Fc^0 := \Bc(\Om^0)$ and $\Fc^1 := \Bc(\Om^1)$ as the Borel $\sigma$-field of respectively $\Om^0$ and $\Om^1$.
	On $\Om^0$ (resp. $\Om^1$), we define $\F^0$ (resp. $\F^1$) as  {the} canonical filtration generated by $X^0$ (resp. $W$),
	and equip $(\Om^0, \Fc^0)$ (resp. $(\Om^1, \Fc^1)$) with the Wiener measure $\P^0_0$ (resp. $\P^1_0$).
	Let $ {\Ur}$  be a bounded Borel subset of $\R^d$, and $\Uc^0$ denote the collection of all $\F^0$-progressively measurable process $\nu: [0, T] \x \Om^0 \longrightarrow {\Ur}$.
	Then, for each initial condition $(t, \xb^0) \in [0,T] \x \Om^0$, we consider a collection $\Pc^0_W(t, \xb^0)$ of probability measures on $\Om^0$:
	\begin{align*}
		\Pc^0_W(t, \xb^0) := \Big\{
			&\P^0 \in \Pc(\Om^0) ~:
			~X^0_s = \xb^0_t+\int_t^s \nu^{\P^0}_r dr + \int_t^s dW^{\P^0}_r, ~s \in [t,T], ~\P^0\mbox{-a.s.}, \\
			&~~~\P^0[X^0_{t\wedge \cdot} = \xb^0_{t \wedge \cdot} ] = 1,~\mbox{where}~ \nu^{\P^0} \in \Uc^0,~W^{\P^0} ~\mbox{is a}~(\P^0, \F^0) \mbox{-Brownian motion}
		\Big\}.
	\end{align*}
	Next, let $\Om := \Om^0 \x \Om^1$, $\Fc := \Fc^0 \otimes \Fc^1 = \Bc(\Om)$, and
	$$
		\Pc_W(t, \xb^0)
		~:=~
		\big\{ \P = \P^0 \x \P^1_0 ~: \P^0 \in \Pc^0_W(t, \xb^0) \big\}.
	$$	
   We also consider the canonical space $(\R^d, \Bc(\R^d))$, equipped with the canonical element $\xi$.
	
	\vspace{0.5em}
	
	We are given a bounded\footnote{As usual, boundedness could be replaced by a suitable growth condition, see for instance \cite[Assumption 2.8]{Djete DPP}.} measurable coefficient  $(\sigma,\sigma_{{0}}) : \R_+ \x \R^d \x \Pc_2(\R^d) \longrightarrow  \mathbb{M}^d \x \mathbb{M}^d$ with $\mathbb{M}^d$ denoting the collection of all $d \x d$ matrix.
	Hereafter, we assume that $(\sigma,\sigma_{{0}})(t,\cdot)$ is Lipschitz continuous, uniformly in $t\le T$.
	Then, for all $t \in [0,T]$, $\mu \in \Pc_2(\R^d)$ and $\P \in \Pc_W(t, \xb^0)$,
	on the space $\R^d \x \Om$ with canonical element $(\xi, X^0, W)$,
	we consider the McKean-Vlasov SDE
	$$
		X_s^{t,\mu,\P}
		=
		\xi + \int_t^s \sigma_0(r, X_r^{t,\mu,\P}, \rho^{t, \mu, \P}_r )dX^0_r +  \int_t^s \sigma(r,X_r^{t,\mu, \P}, \rho^{t, \mu, \P}_r ) dW_r,
		~\mu \x \P \mbox{-a.s.},
	$$
	with $\rho^{t, \mu, \P}_r = \Lc^{\mu \x \P} (X_r^{t,\mu,\P} | \Fc^{X^0}_r )$, where $\Fc^{X^0}_r := \sigma(X^0_{s} ~:{s} \in [0,r])$.
	Notice that the above McKean-Vlasov SDE has a unique strong solution ({see e.g.~\cite[Appendix A]{Djete DPP}}).
	
	\vspace{0.5em}

	Finally, let us define an enlarged canonical space
	$$
		\Omb := \Om \x \Cc([0,T], \R^d) \x \Cc([0,T], \Pc_2(\R^d)),
	$$
	with canonical process $(X^0, W, X, \rho)$, and
	$$
		\Pcb_W(t, \mu)
		~:=~ \Big\{
			\Pb := (\mu \x \P) \circ \big( X^0, W, X^{t, \mu, \P}, \rho^{t, \mu, \P} \big)^{-1}
			~:
			\P \in \Pc_W(t, \xb^0), ~\xb^0 \in \Om^0
		\Big\}.
	$$
   We denote by $\Fb^{X^0} = (\Fcb^{X^0}_t)_{0 \le t \le T}$ the filtration generated by $X^0$ on the enlarged canonical space,
   and denote by $\Pb_0$ the probability measure on $\Omb$ under which the canonical process $X^0$ is a Brownian motion.
	Notice that, for $\overline \P = \mu \times \P^0 \times \P^1_0 \in  {\Pcb_W}(t, \mu)$, the part $\P^1_0$ is the fixed Wiener measure,
	and $\P^0$ belongs to $\Pc^0_W(t, \xb^0)$ under which the canonical process $X^0$ is a diffusion process with  drift $\nu^{\P^0}$.
	By abuse of notation, we denote by $\nu^{\Pb}$ the corresponding drift process of $X^0$ on $\Omb$
   and by $W^{\Pb}$ the corresponding $(\P^0,\F^0)$-Brownian motion part of $X^0$ on $\Omb$
   , i.e.
	$$
		X^0_s = X^0_t + \int_t^s \nu^{\Pb}_r dr + \int_t^s dW^{\Pb}_r, ~s \in [t, T], ~\Pb \mbox{-a.s.}
	$$
	
	\begin{remark}
A probability measure $\Pb \in \Pcb_W(t, \mu)$ describes the distribution of  {the} controlled McKean-Vlasov process with initial distribution $\mu$ at initial time $t$, and control $\nu^{\Pb}$.
		Since controls take values in the bounded set $\Ur$, given a fixed $\mu_0 \in \Pc_2(\R^d)$,
		all the probability measures in the set $\Pcb_W(t, \mu_0)$ are equivalent by Girsanov's theorem.
	\end{remark}
	
	Let $L: [0,T] \x \Pc_2(\R^d) \x {\Ur} \longrightarrow \R$ and $g:  \Pc_2(\R^d) \longrightarrow \R$ be  {measurable} functions,
	 {to which we associate} the value function of the McKean-Vlasov control problem  {through}
	\begin{equation} \label{eq:def_V_valuefunction}
		V(t, \mu) ~:= \sup_{\Pb \in \Pcb_W(t,\mu)} J\left(t, \Pb \right),
		~~\mbox{with}~
		J(t,\Pb) := \E^{\Pb} \Big[\int_{t}^{T} L \big(s, \rho_s, \nu_s^{\Pb} \big)ds + {g}(\rho_T) \Big].
	\end{equation}
	{Here again, we assume for simplicity that $L$ and $g$ are bounded.}

	\vspace{0.5em}

	Given a probability measure $\mu \in \Pc(\R^d)$, and (measurable) functions $\varphi: \R^d \longrightarrow \R$ and $\psi: \R^d \x \R^d \longrightarrow \R$,
	we denote (whenever the integrals are well-defined)
	$$
		\mu(\varphi):=\int_{\R^d} \varphi(x) \mu(d x), \quad \mu \otimes \mu(\psi):=\int_{\R^d \times \R^d} \psi\left(x, x^{\prime}\right) \mu(d x) \mu\left(d x^{\prime}\right) .
	$$
	Let us also define
	\begin{align*}
		\Kc :=   \Big\{\phi: [0,T] \x \Pc_2(\R^d) \x \R^d \rightarrow \R^d
		~:~
		 \phi~\mbox{is bounded and Borel measurable} \Big\} .
	\end{align*}

	\begin{remark}
		The above McKean-Vlasov optimal control problem satisfies the following dynamic programming principle:
		for $(t,\mu) \in [0,T] \x \Pc_2(\R^d)$ and any $\Fb^{X^0}$-stopping time $\tau$ taking values in $[t, T]$, one has
		\begin{align*}
			V(t,\mu)
			~ = ~ & \sup_{\Pb \in \Pcb_{ {W}}(t,\mu)}
			\E^{\Pb}\left[\int_{t}^{\tau}L(s,\rho_s,\nu_s^{\Pb})ds + V(\tau, \rho_{\tau})\right].
		\end{align*}
		Under general conditions, the value function $V$ can be proved  {to be} the unique (viscosity) solution to the HJB master equation (see e.g. Pham and Wei \cite{Pham and Wei}):
		$$
			\left\{\begin{aligned}
			-  \partial_t V(t,\mu)
			-  \L V(t,\mu)
			- H( {t},\mu,  {D_m V})
			&   =   0, \quad(t, \mu) \in[0, T) \times \mathcal{P}_2(\mathbb{R}^d),  \\
			V(T, \mu)
			&   =   g(\mu), \quad \mu \in \mathcal{P}_2(\mathbb{R}^d),
			\end{aligned}\right.
		$$
		where $\L $ is the operator  {defined for} $\phi \in C^2( \Pc_2(\R^d))$  {as}
		\begin{align*}
			\L {\phi}(t,\mu)
			~:=~&
			\mu \Big(
                               \operatorname{tr}(\partial_x D_m \phi(\mu,x)
                               (\sigma\sigma^{\operatorname{T}}+\sigma_0\sigma_0^{\operatorname{T}})(t,x,\mu)) \Big) \\
                               &+~ \mu \otimes \mu \Big(\frac{1}{2} \operatorname{tr}(D_m^{2} \phi(\mu)(x,x^{\prime})\sigma_0(t,x,\mu)\sigma_0^{\operatorname{T}}(t,x^{\prime},\mu)) \Big),
		\end{align*}
		and
       $H: [0,T] \x \Pc_2(\R^d) \x \Kc \longrightarrow \R$
       is the Hamiltonian defined by
		\begin{equation}\label{Hamiltonian}
			H(t,\mu, p)
			~:=~
			\sup_{{u} \in {\Ur}}\left(L(t,\mu,{u}) + {u} \mu \left(\sigma_0(t,\cdot,\mu) p(t, \mu,\cdot)\right)\right).
		\end{equation}
	\end{remark}

\vspace{0.5em}

	 {From now on we} fix an initial distribution $m_0 \in \Pc_2(\R^d)$, and study the optimal control problem $V(0, m_0)$. Our verification argument {goes together with a} duality result, \eqref{eq: dualite} below, in which the   two dual  problems are defined as:
	\begin{align*}
		D_1 := \inf \Big\{
			v_0 \in \R :~
			& v_0 \!+\!\! \int_{0}^{T} \!\! \rho_t \left(\sigma_0(t,\cdot,\rho_t)\phi(t, \rho_t,\cdot)\right) dX_t^0 \\
			&\ge~
			{g}(\rho_T)
			+\! \int_{0}^{T} \!\!\! \Big( L(t,\rho_t,\nu_t^{\Pb}) + \nu_t^{\Pb}\rho_t(\sigma_0(t,\cdot,\rho_t)\phi(t, \rho_t,\cdot)) \Big) dt, \\
			&~~~~~~~~~~~~~~~~~
			\Pb~\mbox{-a.s. for all}~ \Pb \in \Pcb_W(0, m_0), ~\mbox{for some}~\phi \in \Kc \Big\},
	\end{align*}
	and
	\begin{align*}
		D_2 :=\inf \Big\{
			v_0 \in \R :
			~\exists \phi \in \Kc ~\mbox{s.t.}~
			& v_0 +
			\!\! \int_{0}^{T}\!\!\!
			\rho_t \left(\sigma_0(t,\cdot,\rho_t)\phi(t, \rho_t,\cdot)\right) dX_t^0 \\
			& \ge~
			{g}(\rho_T)
			+
			\!\! \int_{0}^{T}\!\!\!
			 H(t,\rho_t,\phi)dt,
			~\Pb_0 \mbox{-a.s.}  \Big\}.
       \end{align*}

	Similar to \eqref{eq:def_C11}, we say that a function
	$F : [0,T] \x \Pc_2(\R^d) \longrightarrow \R$
	belongs to $C^{0,1}([0,T) \x \Pc_2(\R^d))$ if
	$F$ and $D_m F$
	are both (jointly) continuous on $[0,T) \x \Pc_2(\R^d)$, where $D_m F: [0,T) \x \Pc_2(\R^d) \x \R^d \longrightarrow \R^d$ is the partial derivative in the sense that $(m,x) \mapsto D_m F(t,m,x)$ is the derivative of $m \mapsto F(t,m)$ as defined in \eqref{eq:def_DmF}.
	%\bru{\sout{For simplicity, we use $D_m F(m,\cdot)$ to denote $(t,x) \mapsto D_m F(t,m,x)$.}}

	\begin{theorem}\label{thm:Theorem1}
		Assume that the value function $V$ is continuous on $[0,T] \x \Pc_2(\R^d)$, belongs to  $C^{0,1}([0,T) \x \Pc_2(\R^d))$
		and that $D_m V$ is uniformly bounded {on $[0,T) \x \Pc_2(\R^d)$}.
		\\
		 {(i)} Then, one has the duality
		\begin{align}\label{eq: dualite}
			V(0, m_0) = D_1 = D_2.
		\end{align}
		 {(ii)} {Assume that} $\hat{\rm u}: [0,T) \x \Pc_2(\R^d) \longrightarrow {\Ur}$ is a Borel measurable function such that, for all $(s,\mu) \in [0,T)\x \Pc_2(\R^d)$,
		\begin{equation} \label{eq:H_optimal}
			H(s,\mu,D_m V) = L(s,\mu, \hat{\rm u}(s,\mu)) +
			\hat{\rm u}(s,\mu) \mu\left(\sigma_0(s,\cdot,\mu)D_m V(s, \mu,\cdot)\right).
		\end{equation}
		%Assume in addition that there exists an optimal solution $\Pb^* \in \Pcb_W(0, m_0)$ to the control problem $V(0, m_0)$.
		Then, there exists $\Ph \in \Pcb_W(0, {m_0})$ such that ${\nu^{\Ph} = \hat {\rm u}(\cdot, \rho)}$, $d\Ph \x dt$ a.e.,
		and $\Ph$ is an optimal solution to the control problem $V(0, m_0)$, i.e.
		$$
			J(0, \Ph)
			~=~
			V(0, m_0)
			~=~
			\sup_{\Pb \in \Pcb_W(0, m_0)} J(0, \Pb),
		$$
		and
		$$
			V(0, m_0)
			+
			\int_{0}^{T}
			\rho_t\left(\sigma_0(t,\cdot,\rho_t)D_m V(t, \rho_t,\cdot)\right) dX^0_t
			=
			g(\rho_T)
			+ \int_{0}^{T} H(t,\rho_t,D_m V)dt,
			~\hat{\P} \mbox{-a.s.}
		$$
		
	\end{theorem}
	
	Before to provide the proof of the above, let us comment these results.
	
	\begin{remark}
		$\mathrm{(i)}$ In Theorem \ref{thm:Theorem1}, we assume that $D_m V$ is bounded for simplicity in order to apply   Theorem \ref{thm:ItoC1}.
		The same results still hold  if this boundedness condition is replaced by a local integrability condition as in Assumption \ref{assum:second}.
	
		\vspace{0.5em}
		
		\noindent $\mathrm{(ii)}$ When ${\Ur}$ is compact, and $L$ is upper semi-continuous,
		 one can choose a Borel version of the optimizer $\hat{\rm u}$ as required in Theorem \ref{thm:Theorem1}, {see e.g.~\cite[Proposition 7.33, p.153]{BertsekasShreve.78}.}
		
		\vspace{0.5em}
		
		\noindent $\mathrm{(iii)}$ The duality result \eqref{eq: dualite} is in the spirit of  duality results in mathematical finance and optimal transport, see e.g. \cite{Bouchard Dang} for an abstract formulation in optimal control.
		The novelty here is that it {goes together with} the proof of the verification argument and appeals directly to a functional class of controls, which is made possible because we know a priori that $D_{m}V$ is well-defined so that we can identify the optimal control by means of our It\^{o}'s formula for $C^{1}$-functionals.
		
		\vspace{0.5em}
		
		\noindent $\mathrm{(iv)}$
		In the literature of mean-field control or mean-field games, it is usually difficult to check that the value function is $C^2$ (see e.g.~ \cite{Lions, GMMZ} for examples).
		The $C^1$-regularity as required in the above will  be clearly easier to prove.
		We provide in Example \ref{exam:C1} below an example of McKean-Vlasov control problem, in which the $C^1$-regularity can be obtained by using purely probabilistic arguments.

		\vspace{0.5em}
		
		\noindent $\mathrm{(v)}$
		The control problem in \eqref{eq:def_V_valuefunction} is a pure McKean-Vlasov control problem, so that the value function is in the form $V(t, \rho_t)$.
		One could also study  mixed control problems by considering a controlled diffusion process $Y$ in addition to $X$, and a reward function in the form $g(\rho^X_T, Y_T)$, so that the value of the control problem would be in the form $V(t, \rho^X_t, Y_t)$.
		The same arguments as below would lead to a similar verification result.
		We  stay in this pure McKean-Vlasov control setting for ease of presentation.
	\end{remark}

	\begin{remark}
		The reason for considering a weak formulation of the McKean-Vlasov control problem in \eqref{eq:def_V_valuefunction} comes from our duality type arguments which do not apply to strong formulations.
		On the other hand, the  law induced by a strong control is a weak control, as a probability measure in $\Pcb_{{W}}(0, m_0)$.
		Then, given a strong control $\nu$ which achieves the optimality in the Hamiltonian as in \eqref{eq:H_optimal},
		we obtain an optimal control for the weak formulation \eqref{eq:def_V_valuefunction} by Theorem \ref{thm:Theorem1},
		and hence it is also an optimal control for the (more restrictive) strong formulation.
		{More generally speaking}, by a direct adaptation of the arguments in Djete, Possama\"i and Tan \cite[Section 4, Proof of Theorem 3.1]{Djete Equivalence},
		one can prove that, under quite general upper-semicontinuity conditions on $L$ and $g$,
		the value functions of the weak and strong formulations   are the same, even if an optimal strong control does not exist.
	\end{remark}
	
	\begin{proof}[Proof of Theorem \ref{thm:Theorem1}.]
	$\mathrm{(i)}$ Let $(v_0, \phi) \in \R \x \Kc$ be a couple satisfying the inequality in the definition of $D_1$, i.e.
	$$
		v_0 \!+\!\! \int_{0}^{T} \!\! \rho_t \left(\sigma_0(t,\cdot,\rho_t)\phi(t, \rho_t,\cdot)\right) dX_t^0
		\ge
		{g}(\rho_T)
		+\! \int_{0}^{T} \!\!\! \left[ L(t,\rho_t,\nu_t^{\Pb})+ \nu_t^{\Pb} \rho_t(\sigma_0(t,\cdot,\rho_t)\phi(t, \rho_t,\cdot))\right] dt,
		~\Pb~\mbox{-a.s.}
	$$
	for all $\Pb \in \Pc_W(0, m_0)$.
	Taking expectation under $\Pb$ on both sides of the above inequality, it leads to
	$$
		v_0 ~\ge~
		\E^{\Pb} \Big[
			g(\rho_T)
			+
			\int_0^T L(t,\rho_t,\nu_t^{\Pb}) dt
		\Big],
		~\mbox{for all}~ \Pb \in \Pcb_W(0, m_0).
	$$
	Further, by the definition of $H$, we notice that, for all $\Pb \in \Pcb_W(0, m_0)$,
	$$
		L(t,\rho_t,\nu_t^{\Pb}) + \nu_t^{\Pb} \rho_t(\sigma_0(t,\cdot,\rho_t)\phi(t, \rho_t,\cdot))
		~\le~
		 H(t,\rho_t,\phi),
		 ~t\in[0,T].
	$$
	This proves that
	$$
		V(0, m_0) ~\le~ D_1 ~\le~ D_2.
	$$
	
	 {It remains to} prove that $V(0, m_0) \ge D_2$.
	For each $\Pb \in \Pcb_W(0, m_0)$, let us introduce the process $S^{\Pb} = \left(S_t^{\Pb}\right)_{0 \le t \le T}$ by
	$$
		S_t^{\Pb} := V(t,\rho_t) + \int_{0}^{t} L(s,\rho_s,\nu_s^{\Pb}) ds.
	$$	
	For all $0 \le t \le t+h \le T$,  the dynamic programming principle (see e.g. \cite[Theorem 3.2]{Djete DPP})  {implies that}
	\begin{align}
		S_t^{\Pb}(\om)
		~=~&
		\sup_{\tilde{\P} \in  {\bar\Pc}_W(t,\rho_t(\om))}
		\E^{\tilde{\P}}\left[\int_{t}^{t+h}L(s,\rho_s,\nu_s^{\tilde{\P}}) ds
			+ V(t+h, \rho_{t+h})
		\right]
		+
		\int_{0}^{t} L(s,\rho_s,\nu_s^{\Pb})(\om)ds
		\nonumber \\
		~\ge~ &
		\E^{\Pb}\left[\int_{t}^{t+h} L(s,\rho_s,\nu_s^{\Pb}) ds
                   + V(t+h, \rho_{t+h})
                   ~\Big| \Fcb^{X^0}_t \right] (\om)
                   + \int_{0}^{t} L(s,\rho_s,\nu_s^{\Pb})(\om)ds
                   \nonumber \\
		~ = ~ & \E^{\Pb} \left[S_{t+h}^{\Pb}~\Big| \Fcb^{X^0}_t \right](\om),
		\label{supermartingale}
	\end{align}
	 {for $\Pb$-a.e.~$\omega$.}
	In other words, $S^{\Pb}$ is a $(\Pb, \Fb^{X^0})$-supermartingale for all $\Pb \in \Pcb_W(0, m_0)$.
	Therefore, for each $\Pb \in \Pcb_W(0, m_0)$, one can apply the Doob-Meyer decomposition to obtain a unique $\Fb^{X^0}$-predictable non-decreasing process $A^{\Pb}$ and a $(\Pb, \Fb^{X^0})$-martingale $M^{\Pb}$ such that $A^{\Pb}_0 = M^{\Pb}_0 = 0$ and
	$$
		V(t,\rho_t) + \int_{0}^{t} L(s,\rho_s,\nu_s^{\Pb})ds =
		V(0, {m_{0}}) + M_t^{\Pb} - A_t^{\Pb},
		~t \in [0,T],~\Pb \mbox{-a.s.}
	$$
	At the same time, as $V \in C^{0,1}$, we apply our $C^1$-It\^o's formula in Theorem \ref{thm:ItoC1} to deduce another unique decomposition
	$$
		V(t,\rho_t)
		=
		V(0, m_0)
		+
		\int_{0}^{t}\rho_s \big(\sigma_0(s,\cdot,\rho_s) D_m V(s, \rho_s,\cdot) \big) dW_s^{\Pb} + \Gamma_t^{\Pb}, ~\Pb \mbox{-a.s.},
	$$
	where $(\Gamma_t^{ {\Pb}})_{0 \le t \le T}$ is an orthogonal process.
	The above two decompositions are unique, so that the two martingale parts should be the same.
	It follows  {that}
	$$
		V(t,\rho_t) + \int_{0}^{t} L(s,\rho_s,\nu_s^{\Pb})ds =
		V(0, m_0) + \int_{0}^{t}\rho_s\left(\sigma_0(s,\cdot,\rho_s)D_m V(s, \rho_s,\cdot)\right) dW_s^{\Pb} - A_t^{\Pb},
		~ \Pb\mbox{-a.s.}
	$$
       As $V(T, \rho_T) = g(\rho_T)$ and $A^{\Pb}$ is non-decreasing, this implies
	\begin{align}\label{eq:C1 and Doob}
		& V(0,m_0) + \int_{0}^{T}\rho_t \left(\sigma_0(t,\cdot,\rho_t)D_m V(t, \rho_t,\cdot)\right) dX^0_t \nonumber \\
		\ge~ &
		{g}(\rho_T)
		+
		\int_{0}^{T}
		\Big(
			\rho_t \left(\sigma_0(t,\cdot,\rho_t)D_m V(t, \rho_t,\cdot)\right) \nu_t^{\Pb}
			+ L(t,\rho_t,\nu_t^{\Pb})
		\Big) dt .
	\end{align}
	For each $\eps > 0$, one can then
	apply the measurable selection arguments, e.g.~combine \cite[Proposition 2.21]{Karoui Tan} with \cite[Lemma 7.27, p.173]{{BertsekasShreve.78}},
	to obtain $\nu^{\eps} \in \Uc^0$ such that
	$$
		L(s, \rho_s,\nu^{\eps}_s) + \nu^{\eps}_s \rho_s(\sigma_0(s,\cdot,\rho_s){D_{m}V}(s, \rho_s,\cdot))
		~\ge~
		H(s,\rho_s,\phi) - \eps,
		~d\Pb_0 \x dt \mbox{-a.e.}
	$$
	One can then construct a probability measure $\Pb^{\eps}$ such that $\nu^{\Pb^{\eps}} = \nu^{\eps}$.  {As all probability measures in $\Pcb_W(0, m_0)$ are equivalent,}
	  together with \eqref{eq:C1 and Doob},  {this shows} that
	\begin{align*}
		& T \eps + V(0,m_0)
		+
		\int_{0}^{T}\rho_t \left(\sigma_0(t,\cdot,\rho_t)D_m V(t, \rho_t,\cdot)\right) dX^0_t
		\ge~
		{g}(\rho_T)
		+
		\int_{0}^{T}
		H(t,\rho_t, D_m V ) dt ,
		~\Pb_0\mbox{-a.s.}
	\end{align*}
	 {By arbitrariness of} $\eps > 0$,   $V(0, m_0) \ge D_2$.
	
	\vspace{0.5em}
	
	\noindent $\mathrm{(ii)}$  {Assume now that there exists a Borel measurable map} $\hat{\ur}: [0,T) \x \Pc_2(\R^d) \longrightarrow {\Ur}$ such that
	$$
		H(s,\rho_s,D_m V) = L(s,\rho_s, \hat{{\ur}}(s,\rho_s)) + \hat{{\ur}}(s,\rho_s) \rho_s\left(\sigma_0(s,\cdot,\rho_s)D_m V(s, \rho_s,\cdot)\right), ~ s \in [0,T].
	$$
	We can then construct a probability measure $\Ph$ such that $(\nu^{\Ph}_t)_{0 \le t < T} = ({\hat\ur} (t, \rho_t))_{0 \le t < T}$, $\Ph$-a.s.
	To show that $\Ph$ is an optimal solution to the control problem $V(0, m_0)$,  {we appeal to} Lemma \ref{McK-Vlas equality} below  {to deduce that}
	\begin{equation} \label{eq:cliam_eq_H_MKV}
		V(0, m_0)
		 +
		\int_{0}^{T}
		\rho_t\left(\sigma_0(t,\cdot,\rho_t)D_m V(t, \rho_t,\cdot)\right) dX^0_t
		=
		g(\rho_T)
               + \int_{0}^{T} H(t,\rho_t,D_m V)dt,
               ~\hat{\P} \mbox{-a.s.}
	\end{equation}
	We can then compute directly that
	\begin{align*}
		V(0, m_0)
		~ = ~ &
		g(\rho_T)
		+ \int_{0}^{T} H(t,\rho_t,D_m V)dt
		- \int_{0}^{T}
		\rho_t\left(\sigma_0(t,\cdot,\rho_t)D_m V(t, \rho_t,\cdot)\right) dX^0_t \\
		~ = ~ &
		g(\rho_T) +
		\int_{0}^{T} L(s,\rho_s, {\hat\ur}(s,\rho_s))ds -
		\int_0^T
		\rho_t\left(\sigma_0(t,\cdot,\rho_t)D_m V(t, \rho_t,\cdot)\right) dW_t^{\hat{\P}}.
	\end{align*}
	Taking  expectation on both sides under $\Ph$, it follows that
	\begin{align*}
		V(0, m_0)
		~=~
		\E^{\Ph} \left[g(\rho_T) + \int_{0}^{T} L(s,\rho_s,  {\hat\ur}(s,\rho_s)) ds\right]
		~=~
		J(0, \Ph),
	\end{align*}
	i.e. $\Ph$ is an optimal solution to the control problem $V(0, m_0)$.
\end{proof}

	\begin{lemma}\label{McK-Vlas equality}
		In the setting of Theorem $\ref{thm:Theorem1}$  {(ii)}, we have
		\begin{equation*}
			V(0,m_0) + \int_0^T \rho_t\left(\sigma_0(t,\cdot,\rho_t)D_m V(t, \rho_t,\cdot)\right) dX^0_t
			~=~
			g(\rho_T) + \int_0^T H \big(t, \rho_t, D_m V \big) dt,
			~\Pb_0\mbox{-a.s.}
		\end{equation*}
	\end{lemma}

   \begin{proof}
   From the definition of the process $(S^{\Pb}_t)_{0 \le t \le T}$ and its decomposition given in the proof of Theorem \ref{thm:Theorem1}, we have
   \begin{align*}
           \sup_{\Pb \in \Pcb_W(0,m_0)} \E^{\Pb} [S^{\Pb}_T]
       &   ~=~ \sup_{\Pb \in \Pcb_W(0,m_0)} \E^{\Pb}\Big[g(\rho_T) + \int_{0}^{T} L(t,\rho_t,\nu_t^{\Pb})dt\Big] \\
       &   ~=~
		V(0,m_0) - \inf_{\Pb \in \Pcb_W(0,m_0)} \E^{\Pb}[A_{ {T}}^{\Pb}], ~ {\Pb_{0}} \mbox{-a.s.}
   \end{align*}
   which implies that $\inf_{\Pb \in \Pcb_W(0,m_0)} \E^{\Pb}[A_{ {T}}^{\Pb}] = 0$.
    {Moreover, for $\Pb \in \Pcb_W(0,m_0)$} and $\Ph \in \Pcb_W(0 ,m_0)$ such that $\nu^{\Ph} =  {\hat\ur}(\cdot, \rho)$, $d\Pb \x dt$-a.e., we have
   \begin{align}\label{eq:McV optimal equality}
       V(0,m_0)  ~=~ &
                  g(\rho_T) + \int_{0}^{T}
                  \left[L(t,\rho_t,\nu^{\Pb}_t) + \nu^{\Pb}\rho_t\left(\sigma_0(t,\cdot,\rho_t)D_m V(t, \rho_t,\cdot)\right) \right] dt
                  + A^{\Pb}_T \nonumber \\
                & - \int_{0}^{T}\rho_t\left(\sigma_0(t,\cdot,\rho_t)D_m V(t, \rho_t,\cdot)\right) dX^0_t
                   \nonumber\\
                  ~=~ &
                  g(\rho_T) + \int_{0}^{T}
                  \left[L(t,\rho_t,\nu^{\Ph}_t) + \nu^{\Ph}_t \rho_t\left(\sigma_0(t,\cdot,\rho_t)D_m V(t, \rho_t,\cdot)\right)\right] dt
                  + A^{\Ph}_T \nonumber\\
                 & - \int_{0}^{T}\rho_t\left(\sigma_0(t,\cdot,\rho_t)D_m V(t, \rho_t,\cdot)\right) dX^0_t
                 \\
       ~\ge~  &   g(\rho_T) + \int_{0}^{T}
                  \left[L(t,\rho_t,\nu^{\Pb}_t) + \nu^{\Pb}_t \rho_t\left(\sigma_0(t,\cdot,\rho_t)D_m V(t, \rho_t,\cdot)\right)\right] dt
                  + A^{\Ph}_T \nonumber \\
                 & - \int_{0}^{T}\rho_t\left(\sigma_0(t,\cdot,\rho_t)D_m V(t, \rho_t,\cdot)\right) dX^0_t. \nonumber
   \end{align}
    {Combining the above} implies that $0 \le A^{\Ph}_T \le A^{\Pb}_T$ a.s.  {for $\Pb \in \Pcb_W(0,m_0)$}, and
   $$
   0 = \inf_{\Pb \in \Pcb_W(0,m_0)} \E^{\Pb}[A^{\Pb}_T]  \ge \inf_{\Pb \in \Pcb_W(0,m_0)} \E^{\Pb} [A^{\Ph}_T] = 0.
   $$
   At the same time,  {we have by the} reverse  {H\"older}'s inequality
   \begin{align*}
       \inf_{\Pb \in \Pcb_W(0,m_0)} \E^{\Pb}[A^{\Ph}_T] ~=~ &
                       \inf_{\Pb \in \Pcb_W(0,m_0)}
                       \E^{\Pb_0}\left[\frac{d\Pb}{d\Pb_0} A^{\Ph}_T \right] \\
       ~\ge~ &         \inf_{\Pb \in \Pcb_W(0,m_0)}
                       \frac{\E^{\Pb_0}\left[(A^{\Ph}_T)^{\frac{1}{2}}\right]^2}
                       {\E^{\Pb_0}\left[\left(\frac{d\Pb}{d\Pb_0}\right)^{-1}\right]} \\
       ~\ge~ &         \frac{\E^{\Pb_0}\left[(A^{\Ph}_T)^{\frac{1}{2}}\right]^2}{C},
   \end{align*}
   where $C > 0$ is a fixed constant such that $\E^{\Pb_0}[(\frac{d\Pb}{d\Pb_0})^{-1}] \le C$ for $\forall ~ \Pb \in \Pcb_W(0,m_0)$ {whose existence is justified as follows:  }
   \begin{align*}
       \left(\frac{d\Pb}{d\Pb_0}\right)^{-1} ~=~ &
                   \exp \left(\int_{0}^{T} -\nu^{\Pb}_t dW^{\Pb_0}_t
                       + \frac{1}{2}\int_{0}^{T} (\nu^{\Pb}_t)^2 dt \right) \\
       ~=~ &       \exp \left(\int_{0}^{T} -\nu^{\Pb}_t dW^{\Pb_0}_t
                       - \frac{1}{2}\int_{0}^{T} (\nu^{\Pb}_t)^2 dt \right)
                   \exp \left(\int_{0}^{T} (\nu^{\Pb}_t)^2 dt \right),
   \end{align*}
   which implies that
   $$
   \E^{\Pb_0}\left[\left(\frac{d\Pb}{d\Pb_0}\right)^{-1}\right]
   \le
   \exp\left(T\bar{{u}}^2 \right) =: C,
   $$
   {with $\bar{{u}} := \max_{{u} \in {\Ur}}|{u}|$.}
   Therefore,  {combining the above shows that} $A^{\Ph}_T = 0, \Pb_0$-a.s., and we can conclude by \eqref{eq:McV optimal equality} that
   $$
       V(0,m_0) + \int_0^T \rho_t\left(\sigma_0(t,\cdot,\rho_t)D_m V(t, \rho_t,\cdot)\right) dX^0_t
		~=~
		g(\rho_T) + \int_0^T H \big(t, \rho_t, D_m V \big) dt,
		~\Pb_0\mbox{-a.s.}
   $$
   \end{proof}

	\begin{example} \label{exam:C1}
		Let $d \ge 1$, $\sigma (\cdot) = \sigma_0(\cdot) \equiv 1$, $\Ur$ be a convex and compact subset of $\R^d$, $L(t, \mu, u) = \bar L(u)$ for some function
		$$
			\bar L: \Ur \longrightarrow \R~\mbox{, strictly concave},
		$$
		and
		$g(\mu) = \bar g \big( \mu(\phi) \big)$ with $\bar g: {\R^{d}} \longrightarrow {\R} $ (resp. $\phi: {\R} \longrightarrow {\R^{d}}$)  in $C^1_b({\R^{d}})$ (resp. $C^1_b ({\R})$).
		In this simple setting, we can re-write the value function as
		\begin{equation} \label{eq:V_reformulation}
			V(t, \mu) ~= \sup_{\P^0 \in \Pc^0_W(t,0)} \Jb (t, \mu, \P^0),
		\end{equation}
		where
		\begin{equation} \label{eq:def_Jb}
			\Jb (t, \mu, \P^0) ~:=~ \E^{\P^0} \Big[  \int_t^T \bar L \big( \nu^{\P^0}_s \big) ds + \bar g \Big( \int_{\R^d} \bar \phi_t (z + X^0_T)  \mu (dz) \Big) \Big],
		\end{equation}
		with
		$$
			\bar \phi_t (y ) ~:=~  \E^{\P^1_0} \big[ \phi(y + W_T - W_t) \big],
		$$
		and
		\begin{align*}
			\Pc^0_W(t, 0)
			&:=
			\Big \{ \P^0 \in \Pc(\Om^0) ~: X^0_s = \int_t^{s \vee t} \nu^{\P^0}_r dr + \int_t^s dW^{\P^0}_r, ~ s \in [0,T], ~\P^0\mbox{-a.s.} \\
			&~~~~~~~~~~~~~~~~~~~~~~~~
			\mbox{where}~\nu^{\P^0} \in \Uc^0, ~ W^{\P^0} ~\mbox{is a}~ (\P^0, \F^0) \mbox{-Brownian motion}
			\Big\}.
		\end{align*}
		For fixed $\P^0 \in \Pc^0_W(t, 0)$, it is clear that the map $\mu \longmapsto \Jb(t, \mu, \P^0)$ is differentiable and
		\begin{equation} \label{eq:formula_DmV}
			D_m \Jb(t, \mu, x, \P^0)
			~=~
			\E^{\P^0} \Big[ \bar g' \Big( \int_{\R^d} \bar \phi_t (z + X^0_T) \mu(dz) \Big) \nabla \bar \phi_t \big(x + X^0_T \big) \Big]
		\end{equation}
	{in which $ \bar g' $ is  the Jacobian of $\bar g$ and $\nabla \bar \phi_t$ is the gradient of $\bar \phi_{t}$ as a column vector.}
		Clearly, $D_m \Jb$ is continuous in all its arguments.
		
		\vspace{0.5em}
		
		Next, as $L$ depends only on $u$, we can apply Tan and Touzi \cite[Lemma 3.9]{TanTouzi} to deduce that the map
		$$
			\P^0 \longmapsto \Jb (t, \mu, \P^0)
			~\mbox{is upper semicontinuous,}
		$$
		while, since $\Ur$ is convex and compact,   $\Pc^0_W(t,0)$ is also convex and compact (for the weak convergence topology).
		Then, for every fixed $(t, \mu)$, there exists an optimizer for the optimal control problem \eqref{eq:V_reformulation}.
		
		\vspace{0.5em}
		
		We now prove that
		\begin{equation} \label{eq:Jb_concave}
			\P^0 \longmapsto \Jb(t, \mu, \P^0)
			~\mbox{is strictly concave.}
		\end{equation}
		Let $\P^0_1, ~\P^0_2 \in \Pc^0_W(t, 0)$ and $\P^0_3 := (\P^0_1 + \P^0_2) /2$.
		Following the arguments in the proof of \cite[Proposition 3.11.(ii) and Lemma 3.15 ]{TanTouzi},
		there exists an enlarged canonical space $\overline \Om^0$ (of $\Om^0 := \Cc([0,T], \R^d)$) with canonical process $(X^0, \bar \nu)$,
		together with $\Pb^0_1, \Pb^0_2 \in \Pc(\Omb^0)$ such that
		$$
			\P^0_i \circ (X^0, \nu^{\P^0_i})^{-1}
			~=~
			\Pb^0_i \circ (X^0, \bar \nu)^{-1},  ~i = 1, 2,
		$$
		so that
		$$
			\E^{\P^0_i} \Big[  \int_t^T \bar L \big( \nu^{\P^0_i}_s \big) ds \Big]
			~=~
			\E^{\Pb^0_i} \Big[  \int_t^T \bar L \big( \bar \nu_s \big) ds \Big],  ~i = 1, 2.
		$$
		Let $\Pb^0_3 := (\Pb^0_1 + \Pb^0_2) /2$ so that $\Pb^0_3 |_{\Om^0} = \P^0_3$.
		Since $U$ is convex,  one can apply the classical projection theorem (see e.g. \cite[Theorem A.3]{TanTouzi}) to deduce that
		$\P^0_3 \in \Pc^0_W(t, 0)$ and
		$$
			\nu^{\P^0_3}_s = \E^{\Pb^0_3} \big[ \bar \nu_s | \Fcb^{X^0}_s \big], ~d \P^0_3 \x d s \mbox{-a.e.},
		$$
		in which $\Fcb^{X^0}_s$ is the $\sigma$-field on the enlarged space $\Omb^0$ generated by $X^0_{s \wedge \cdot}$,
		and we consider $\E^{\Pb^0_3} \big[ \bar \nu_s | \Fcb^{X^0}_s \big]$ as a random variable in $\Om^0$. {Moreover, $\nu^{\P^0_3}_{s}=  \bar \nu_{s}$  $d \Pb^0_3 \x d s \mbox{-a.e.}$ only if $\P^0_1 = \P^0_2$.}
		Since $\bar L$ is strictly concave, when $\P^0_1 \neq \P^0_2$, it follows by Jensen's inequality that
		\begin{align*}
			\frac12 \Big( \E^{\P^0_1} \Big[  \int_t^T \bar L \big( \nu^{\P^0_1}_s \big) ds \Big] +  \E^{\P^0_2} \Big[  \int_t^T \bar L \big( \nu^{\P^0_2}_s \big) ds \Big] \Big)
			=~&
			\frac12 \Big( \E^{\Pb^0_1} \Big[  \int_t^T \bar L \big( \bar \nu_s \big) ds \Big] +  \E^{\Pb^0_2} \Big[  \int_t^T \bar L \big( \bar \nu_s \big) ds \Big] \Big)\\
			=~&
			\E^{\Pb^0_3} \Big[  \int_t^T \bar L \big( \bar \nu_s \big) ds \Big] \\
			<~&
			{\E^{\Pb^0_3} \Big[  \int_t^T \bar L \big( \nu^{\P^0_3}_s \big) ds \Big]}\\
			=~& {\E^{\P^0_3} \Big[  \int_t^T \bar L \big( \nu^{\P^0_3}_s \big) ds \Big]}.
		\end{align*}
		As $\P^0_3:= (\P^0_1 + \P^0_2) /2$, this implies that
		$$
			\P^0 \longmapsto \E^{\P^0} \Big[  \int_t^T \bar L \big( \nu^{\P^0}_s \big) ds \Big]
			~\mbox{is strictly concave}.
		$$
		Since
		$$
			\P^0 \longmapsto \E^{\P^0} \Big[ \bar g \Big( \int_{\R^d} \bar \phi_t (z + X^0_T)  \mu (dz) \Big) \Big]
			~\mbox{is linear},
		$$
		 it follows by the definition of $\Jb$ in \eqref{eq:def_Jb} that \eqref{eq:Jb_concave} holds true.
		Therefore, for every fixed $(t, \mu)$, there exists a unique optimizer $\Ph^0_{t,\mu}$ for the optimal control problem \eqref{eq:V_reformulation}.
		In particular, one immediately deduces that
		\begin{equation} \label{eq:continuity_Ph}
			(t, \mu) ~\longmapsto~ \Ph^0_{t,\mu} ~\mbox{is continuous}.
		\end{equation}

		Now, we claim that
		\begin{equation} \label{eq:DV_DJ}
			\frac{\delta V}{\delta m} (t, \mu, x)
			~=~
			\frac{\delta \Jb}{\delta m} (t, \mu, x, \Ph^0_{t,\mu}),
			~~\mbox{for all}~
			(t,\mu, x).
		\end{equation}
		Indeed, given $\mu \neq \mu'$, and with the notation $\mu_{\lambda} := \lambda \mu + (1-\lambda) \mu'$, one has
		$$
			V(t, \mu) - V(t, \mu') \le \Jb(t, \mu, \Ph^0_{t,\mu}) - \Jb (t, \mu', \Ph^0_{t, \mu})
			=
			\int_0^1 \int_{\R^d} \frac{\delta \Jb}{\delta m}  \big( t, \mu_\lambda, x, \Ph^0_{t, \mu} \big) ( \mu - \mu') (dx) d\lambda,
		$$
		and similarly
		$$
			V(t, \mu) - V(t, \mu') \ge \int_0^1 \int_{\R^d} \frac{\delta \Jb}{\delta m} \big( t, \mu_\lambda, x, \Ph^0_{t, \mu'} \big) ( \mu - \mu') (dx) d\lambda.
		$$
		Using the continuity of $(t, \mu) \mapsto \Ph^0_{t, \mu}$ in \eqref{eq:continuity_Ph} {and applying the two inequalities just above to $\mu'=\mu'_{\eps}:=\mu+\eps(\mu''-\mu)$ with $\eps\downarrow 0$}, this is enough to prove  \eqref{eq:DV_DJ}.
		Taking derivatives on both sides in \eqref{eq:DV_DJ},
		one has
		$$
			D_m V (t, \mu, x)
			~=~
			D_m \Jb (t, \mu, x, \Ph^0_{t,\mu}).
		$$
		Finally, using \eqref{eq:continuity_Ph} together with the continuity of $D_m \bar J$, we deduce  that $V \in C^{0,1} ([0,T] \x \Pc_2(\R^d))$.
	\end{example}

\end{document}